\documentclass[a4paper,11pt]{amsart}
\usepackage[top=2.2cm,left=2.8cm,right=2.8cm,bottom=2.8cm]{geometry}

\usepackage[foot]{amsaddr}

\usepackage{hyperref}
\hypersetup{
    colorlinks,
    citecolor=green,
    filecolor=black,
    linkcolor=blue,
    urlcolor=blue
}
\hypersetup{linktocpage}

\usepackage{hhline}

\usepackage{subfloat}
\usepackage{mathtools}
\usepackage{cite}
\usepackage{graphicx}
 
 \usepackage{cleveref}
\usepackage{amsbsy}
\usepackage{latexsym}
\usepackage{amsfonts}
\usepackage{amssymb}
\usepackage{upgreek}
\usepackage[usenames]{color}
\usepackage{amsmath,amsthm}
\usepackage{enumerate}
\usepackage{stmaryrd}

\usepackage{multirow}

\usepackage{pgfplots} 
\usepackage{tikz}
\usepackage{tikzscale}
\usetikzlibrary{shapes,backgrounds,calc}

\usetikzlibrary{decorations.pathreplacing}
\usetikzlibrary{arrows.meta}
\usepackage[caption=false]{subfig}
\definecolor{backcolor}{rgb}{.7,.7,1}
\definecolor{backcolor2}{rgb}{1,.7,0.7}
\usetikzlibrary{positioning}
\makeatletter
\tikzset{circle split part fill/.style  args={#1,#2}{%
		alias=tmp@name, 
		postaction={%
			insert path={
				\pgfextra{%
					\pgfpointdiff{\pgfpointanchor{\pgf@node@name}{center}}%
					{\pgfpointanchor{\pgf@node@name}{east}}%
					\pgfmathsetmacro\insiderad{\pgf@x}
					\fill[#1] (\pgf@node@name.base) ([xshift=-\pgflinewidth]\pgf@node@name.east) arc
					(0:180:\insiderad-\pgflinewidth)--cycle;
					\fill[#2] (\pgf@node@name.base) ([xshift=\pgflinewidth]\pgf@node@name.west)  arc
					(180:360:\insiderad-\pgflinewidth)--cycle;            
}}}}}  
\makeatother

\tikzset{
	position/.style args={#1:#2 from #3}{
		at=(#3), anchor=#1+180, shift=(#1:#2)
	}
}

\DeclareMathOperator{\Div}{div}
\DeclareMathOperator*{\argmin}{arg\,min}
\DeclareMathOperator*{\argmax}{arg\,max}

\begin{document}
\title[Variationally Correct Neural Residual Regression]{Variationally Correct Neural Residual Regression for Parametric PDEs: On the Viability of Controlled Accuracy }
\author{Markus Bachmayr$^1$}
\email{bachmayr@igpm.rwth-aachen.de}
\author{Wolfgang Dahmen$^{1,2}$}
\email{dahmen@igpm.rwth-aachen.de}
\author{Mathias Oster$^{1}$} 
\email{oster@igpm.rwth-aachen.de}

\address{$^1$ Institut f\"ur Geometrie und Praktische Mathematik, RWTH Aachen University, Templergraben 55, 52062 Aachen, Germany}
\address{$^2$ University of South Carolina, Mathematics Department, 1523 Greene Street, Columbia, SC 29208, USA}

\date{\today}

\thanks{The authors acknowledge funding by the Deutsche Forschungsgemeinschaft (DFG, German Research Foundation) -- project number 442047500 -- through the Collaborative Research Center ``Sparsity and Singular Structures'' (SFB 1481) and by the NSF Grants DMS 2038080, DMS-2012469, DMS-2245097.}

\newtheorem{prop}{Proposition}[section]
\newtheorem{lemma}[prop]{Lemma}
\newtheorem{cor}[prop]{Corollary}
\newtheorem{theorem}[prop]{Theorem}
\newtheorem{assumption}[prop]{Assumption}
\newtheorem{definition}[prop]{Definition}
\newtheorem{property}[prop]{Property}
\theoremstyle{remark}
\newtheorem{rem}[prop]{Remark}
\newtheorem{example}[prop]{Example}

\def\N{\mathbb{N}}
\def\R{\mathbb{R}}

\def\cF{\mathcal{F}}
\def\bbf{\mathbf{f}}
\def\cD{\mathcal{D}}
\def\cI{\mathcal{I}}
\def\cM{\mathcal{M}}
\def\cT{\mathcal{T}}
\def\cV{\mathcal{V}}
\def\cL{\mathcal{L}}
\def\cB{\mathcal{B}}
\def\cS{\mathcal{S}}
\def\cP{\mathcal{P}}
\def\cQ{\mathcal{Q}}
\def\cR{\mathcal{R}}
\def\cU{\mathcal{U}}
\def\bL{\mathbf{L}}
\def\bK{\mathbf{K}}
\def\bC{\mathbf{C}}
\def\X{X\in\{L,R\}}
\def\D{{\Delta}}
\def\H{\mathcal{H}}
\def\bM{\mathbf{M}}
\def\bQ{\mathbf{Q}}
\def\bG{\mathbf{G}}
\def\bP{\mathbf{P}}
\def\bW{\mathbf{W}}
\def\bT{\mathbf{T}}
\def\bV{\mathbf{V}}
\def\bv{\mathbf{v}}
\def\bz{\mathbf{z}}
\def\bw{\mathbf{w}}
\def\bQ{\mathbf{Q}}

\def\t{\tilde}
\def\lll{\langle}
\def\rr{\rangle}

\newcommand{\epsref}{\epsilon_{\mathrm{ref}}}
\newcommand{\epspred}{\epsilon_{\mathrm{pred}}}
\newcommand{\ba}{\mathbf{a}}
\newcommand{\bb}{\mathbf{b}}
\newcommand{\bc}{\mathbf{c}}
\newcommand{\bd}{\mathbf{d}}
\newcommand{\bs}{\mathbf{s}}
\newcommand{\bff}{\mathbf{f}}
\newcommand{\bp}{\mathbf{p}}
\newcommand{\bg}{\mathbf{g}}
\newcommand{\bq}{\mathbf{q}}
\newcommand{\br}{\mathbf{r}}
\newcommand{\PP}{\mathbb{P}}
\newcommand{\G}{\mathbb{G}}
\newcommand{\e}{\varepsilon}

\newcommand{\be}{\begin{equation}}
\newcommand{\ee}{\end{equation}}

\newcommand\eref[1]{(\ref{#1})}

\newcommand\bA{\mathbf{A}}
\newcommand\bB{\mathbf{B}}
\newcommand\bR{\mathbf{R}}
\newcommand\bD{\mathbf{D}}
\newcommand\bE{\mathbf{E}}
\newcommand\bF{\mathbf{F}}
\newcommand\bH{\mathbf{H}}
\newcommand{\cN}{\mathcal N}
\newcommand\bU{\mathbf{U}}
\newcommand\cH{\mathcal{ H}}
\newcommand\sk{{sk}}
\newcommand{\XX}{\mathbb{X}}
\newcommand{\id}{{\rm id}}

\newcommand{\dI}{\Delta}
\newcommand{\cY}{\mathcal{Y}}
\newcommand{\U}{\mathbb{U}}
\newcommand{\V}{\mathbb{V}}
\newcommand{\W}{\mathbb{W}}
\newcommand{\pP}{\mathbb{P}}
\newcommand{\Y}{\mathbb{Y}}

\newcommand{\E}{\mathbb{E}}
\newcommand{\cX}{\mathcal{X}}
 
\newcommand{\dN}{\mathcal{N}\!\mathcal{N}}
\newcommand{\cZ}{\mathcal{Z}}
\newcommand{\tr}{{\rm train}}
\newcommand{\hx}{{\hat{x}}}
\renewcommand{\H}{\mathbb{H}}
\newcommand{\s}{\mathsf{s}}

\newcommand{\wt}{\widetilde}
\renewcommand{\X}{\XX}

\newcommand{\mA}{\mathfrak{A}}
\newcommand{\mI}{\mathfrak{I}}
\newcommand{\mS}{\mathfrak{S}}
\newcommand{\mR}{\mathfrak{R}}
\newcommand{\mP}{\mathfrak{P}}
\newcommand{\mT}{\mathfrak{T}}

\newcommand{\wtF}{{\cF}}
\newcommand{\wtX}{{\X}}
\newcommand{\wtY}{{\Y}}
\newcommand{\wtB}{{\cB}}
\newcommand{\wh}{\widehat}

\newcommand{\pp}{p} %
\newcommand{\pdom}{\cY}%

\newcommand{\uN}{{u^{\mbox{\tiny{$\mathcal{N}\!$}\! }}}}
\newcommand{\uNe}{{u_\e^{\mbox{\tiny{$\mathcal{N}\!$}\! }}}}
\newcommand{\uNi}{{u_i^{\mbox{\tiny{$\mathcal{N}\!$}\! }}}}
\newcommand{\bn}{\mathbf{n}}
\newcommand{\aN}{{a^{\mbox{\tiny{$\mathcal{N}\!$}\! }}}}

\renewcommand{\mP}{\Theta}

\providecommand{\abs}[1]{\lvert#1\rvert}
\providecommand{\bigabs}[1]{\bigl\lvert#1\bigr\rvert}
\providecommand{\Bigabs}[1]{\Bigl\lvert#1\Bigr\rvert}
\providecommand{\biggabs}[1]{\biggl\lvert#1\biggr\rvert}
\providecommand{\lrabs}[1]{\left\lvert#1\right\rvert}

\providecommand{\norm}[1]{\lVert#1\rVert}
\providecommand{\bignorm}[1]{\bigl\lVert#1\bigr\rVert}
\providecommand{\Bignorm}[1]{\Bigl\lVert#1\Bigr\rVert}
\providecommand{\biggnorm}[1]{\biggl\lVert#1\biggr\rVert}
\providecommand{\lrnorm}[1]{\left\lVert#1\right\rVert}

\providecommand{\floor}[1]{\lfloor#1\rfloor}
\providecommand{\ceil}[1]{\lceil#1\rceil}

\newcommand{\md}{p}
\newcommand{\mD}{P}

\newcommand{\loss}{E}
\newcommand{\lossemp}{\loss_{\rm emp}}

\newcommand{\ug}{\ensuremath{\mathfrak{u}}}
\newcommand{\wg}{\ensuremath{\mathfrak{w}}}

\newcommand{\TODO}[1]{\textcolor{blue}{[#1]}}
\newcommand{\dw}[1]{\textcolor{blue}{#1}}
\maketitle

\begin{abstract}
This paper is about learning the parameter-to-solution map for systems of partial differential equations (PDEs) that depend on a potentially large number of parameters 
  covering all PDE types for which a {\em stable variational formulation} (SVF) can be found. 
A central constituent is the notion of {\em variationally correct residual} loss function meaning that its value is always uniformly proportional to the squared solution error in the norm determined by the SVF, hence facilitating  rigorous {\em a posteriori} accuracy control. 
It is based on a single variational problem, associated with the family of parameter dependent fiber problems, employing the notion of  {\em direct integrals of Hilbert spaces}. Since in its original form the 
loss function is given as a dual test norm of the residual a central objective   
is to develop equivalent computable expressions. A first critical role is  played by 
{\em hybrid hypothesis classes}, whose elements are piecewise polynomial in (low-dimensional) spatio-temporal variables with parameter-dependent coefficients that can be represented, e.g. by neural networks. 
Second, working with  first order SVFs, we distinguish  two scenarios: (i)  the test space can be chosen as an $L_2$-space (e.g. for elliptic or parabolic   problems) so that  residuals live in $L_2$ and can be evaluated directly; (ii)
when trial and test spaces for the fiber problems (e.g.   for transport equations) depend
on the parameters, we use {\em ultraweak} formulations. 
In combination with {\em Discontinuous Petrov Galerkin} concepts the hybrid format is 
then instrumental to arrive at  variationally correct   computable residual loss functions. Our findings are illustrated by numerical experiments representing (i) and (ii), namely
elliptic boundary value problems with piecewise constant diffusion coefficients and 
pure transport equations with parameter dependent convection field.

\end{abstract}

\section{Introduction}\label{sec:intro}
The problem of efficient computational exploration of  families of solutions of parameter-dependent partial differential equations, also called \emph{solution manifolds}, arises in many different applications.
Elements of solution manifolds usually describe the viable states of a physical system,
and one is often interested in the selection of parameters
that best explains given measurements or observational data. 
Constructing an efficient approximation of the map taking observational data to
the states in the manifold is therefore at the heart of many forward and inverse 
simulation tasks. 
Approximations of this type can be obtained, for example, by means of projection-based
model order reduction methods or by operator learning methodologies based
 on deep neural networks.

To explain this in more detail, 
we now assume as our starting point a  family of partial differential equations (PDEs), given in general form by some operator $\cF$ depending on parameters 
$\md$ from a domain $\mD$. Given $\md\in\mD$ we then seek a solution $u$
in some Banach space $\U$, called the \emph{trial space}, such that
\be 
\label{par}
\cF(u;\md)=0. 
\ee
Thus, $u$ is not only a function of {\em spatio-temporal variables} in the domain on which the PDE is posed, but also of the {\em parametric variables} $\md\in\mD$.
The choice of function space $\U$ that accommodates   
solutions $u$ depends on $\cF$ and determines the metric in which we measure accuracy of approximations to $u$. For instance, for stationary PDE problems, $\U$ is typically a function space (such as a Sobolev space) on a spatial domain $\Omega$. For time-dependent problems $\Omega$ would typically be a space-time cylinder.

In general, the parameter dependence may reflect incompleteness of the model that underlies \eref{par}. 
For instance, such models may be based on constitutive laws that are not precisely known or depend on missing data.
Assuming that all possibilities of interest are covered by parameters in $\mD$, we can assume that
\eref{par} determines a {\em unique} solution $u=u(\md)\in \U$ for each parameter $\md$ (for an alternative regularization concept for incomplete models, see \cite{DSW,DMS}).
The \emph{solution operator} $\cS: \md \mapsto u(\md)$ is then well-defined, and its range 
\[ 
\cM=\cM(\mD):= \{ u(\md):\md\in\mD\}
\] 
is commonly referred to as the solution manifold associated to the parametric problem.

We postpone a brief discussion of established methods for approximating $\cS$ (and hence $\cM$) to Section \ref{ssec:proj}, but we remark here that 
the universality of deep neural networks (DNNs) and their prominent role in modern machine learning offer promising perspectives for learning-based approximations of solution manifolds. An efficient exploration of $\cM$ then requires the ability to efficiently evaluate $\cS$, which naturally favors approximate representations of $\cS$
that depend on as few degrees of freedom as possible. In this sense, a computationally handy surrogate 
for $\cS$ can be viewed as a {\em reduced model}.

A guiding theme in this work is that
 {\em accuracy quantification} is an indispensable prerequisite 
for prediction capability of a simulation task. Moreover, doing so with respect to metrics that are intrinsic to the model is particularly important when dealing with inverse problems. However, quantifying errors for the high-dimensional
approximation problems with DNNs in operator learning is 
a widely open problem. One major obstacle is that the only practically relevant way
of constructing approximations by DNNs is via non-convex optimization, which
entails a significant uncertainty in optimization success. As a consequence, the theoretically valid expressive power of a given DNN model may not be practically realizable. From the perspective of model reduction, it is  not an option to resort to generous {\em over-parameterizations} of
DNNs that is common in machine learning, so as to drive the underlying loss function to zero and thus ensuring that local minima are actually global ones. 

To cope with these issues from the viewpoint of accuracy quantification, the central
goal   
in this article is to develop loss functions
for the underlying learning problem that we call {\em variationally correct}. This means:
\begin{equation}\tag{VC}\label{VC}
\hspace{-1.2cm}
\begin{minipage}{0.9\textwidth}
\begin{quote}
At any stage of an associated optimization process, the current loss is uniformly
proportional to the squared error of the corresponding estimator with respect
to a metric that is adapted to the problem.
\end{quote}
\end{minipage} 
\end{equation}
 Thus, although one may not guarantee in practice that the optimization
gets close to a minimizer, the terminal loss becomes a rigorous 
a posteriori error bound certifying 
the result. 
The concepts proposed in this paper do not target a specific PDE model but are designed 
to cover a possibly wide scope of PDE types covering, for example,  dissipative elliptic or parabolic problems, transport problems, but also dispersive models.

While so-called purely data-driven regression in the trial space $\U$ is by definition
variationally correct, its significant computational cost motivates looking for residual-type loss functions that avoid the computation of many high-fidelity training samples. A key ingredient in rendering such loss functions variationally correct are {\em stable variational formulations} for \eref{par}. 
This readily leads to idealized formulations in terms
 of residuals that are, however, to be  measured generally in dual norms whose   practical evaluation   is generally problematic. 
We highlight different basic strategies towards computable variationally correct residual loss functions, focusing mainly
on the combination of deep networks with {\em first order system least-squares} 
and with \emph{discontinuous Petrov-Galerkin} (DPG) formulations for different types of parameter-dependent PDEs. 

Concerning the connection to operator learning,
our focus here is on the {\em parameter-to-solution operator} where the parameter dimension $d_\pp$ can be large but remains finite. In general, operator learning  addresses instead situations where $\cS$ acts
on (compact) sets of functions of input data. Corresponding meta-architectures for representing $\cS$ include therefore a first ``discretization layer'' that generates finite-dimensional
input parameters for a subsequent layer that approximates $\cS$. Of course,
when increasing  accuaracy demands the dimension of input parameters generally has to grow.  Mainly in order to avoid an additional level of technicalities, we have skipped here such a discretization step and view $\cS$ as a mapping on a parameter domain of potentially high but finite dimension.
Our priority here is to learn $\cS$ with respect to the correct model-compliant metrics, as explained later in technical terms, which is of particular importance for stability estimates in inverse problems.

\subsection{Beyond projection-based model reduction}\label{ssec:proj}

Finding computationally efficient surrogates for $\cM$ or $\cS$ can be viewed as a task of \emph{model order reduction}. Well-established strategies for achieving this include the {\em reduced basis method} (RBM; see, for example, \cite{RHP,Ro,DPW,BCDDPW}) and {\em proper orthogonal decomposition} (POD; see, 
 for example, \cite{KV}).  
These methods are based on generating a {\em linear} (or affine) subspace of $\U$ such that $\cS$ can be replaced for each parameter $\pp$ by a suitable projection (in the case of the RBM, a Galerkin projection) to this subspace.
Such methods can in principle offer some degree of error control; while in the RBM, one usually aims at approximation in $L_\infty(\mD;\U)$, which can be difficult to achieve in practice for high-dimensional parameter domains (see \cite{CDDN}), POD is particularly suitable for controlling errors in $L_2(\mD;\U)$.

For constructing the corresponding subspace, in both cases one accepts a major computational {\em offline cost} that has to be spent only once, while subsequent parameter queries at an \emph{online} stage can be carried out very efficiently -- in particular, much faster than a single high-fidelty solve of the PDE. To achieve this efficiency, the dimension of the reduced space required for a certain error needs to be small. 
In the case of RBM, a lower bound on the error achievable by an $n$-dimensional subspace is given by the 
{\em Kolmogorov $n$-widths}  
\be
 \label{Kolm}
d_n(\cM)_\U := \inf_{{\rm dim}\,\W=n}\sup_{u\in\cM}\inf_{w\in \W}\|u-w\|_\U.
\ee
For many elliptic problems where $\cS$ depends holomorphically on the parameters, these $n$-widths decay rapidly even for large parametric dimensions due to anisotropies in the parametric dependence (see, e.g., \cite{CDacta} and the literature cited there) or due to additional structural features \cite{BC:17}.

Unfortunately, for PDEs with little or no dissipative effects, such as hyperbolic problems, this is no longer the case, and methods based on a selection of linear subspaces become very inefficient. In particular, for solutions that exhibit jump discontinuities (such as shocks) with parameter-dependent locations, the $n$-widths generally decay only algebraically, with a rate that deteriorates exponentially with respect to the dimension of $\mD$.
Many other methods that build explicit representations of parameter-dependent solutions, such as sparse polynomial expansions in $p$ or low-rank tensor representations (see, e.g., \cite{BCD:18} and the references given there), are subject to the same restrictions in terms of the decay of $n$-widths or related quantities, since they eventually also rely on choices of linear subspaces.

A natural question is thus whether in such cases, $\cM$ still exhibits some sort of structural sparsity that can be exploited by other types of {\em nonlinear reduced models}.
DNNs can provide the required flexibility.
In particular, the solution manifold of pure {\em linear transport equations} with parameter-dependent convection field has been shown to be approximable (under certain structural assumptions on the convection field, such as affine
parameter dependence) by certain DNNs without a curse of dimensionality \cite{Dtransport}. Hence, although the use of DNNs by itself does not guarantee good performance in practice, such nonlinear approximations by compositions have the potential to avoid the basic restrictions of the $n$-widths \eqref{Kolm}.

\subsection{Model reduction by empirical risk minimization}\label{ssec:empiricalrisk}

A common, purely data-driven approach to approximating  $\cS$ (pursued, for example, in \cite{Stu1,DDP,DWW,Stu2}) is based on generating
a training set of synthetic data $\wt u(\md)\approx u(\md)$ for
parameter samples $\md\in \wh\mD$, where $\wh\mD\subset\mD$ is a sufficiently large yet finite subset. Thus, acquiring synthetic data means   approximately solving \eref{par} for each $\md\in \wh\mD$. 
These synthetic data are then used to construct a surrogate for $\cS$ by regression, as an instance of what has become known as \emph{operator learning}. 
A typical way of training
a surrogate $\uN(\cdot,\md;\theta^*)\approx u(\cdot;\md)$, as a function of spatio-temporal variables as well as parametric variables representing $\md$, is to minimize a mean-square {\em empirical risk} functional
\be
\label{datadriv}
\|\wt u- \uN(\theta)\|^2_{\ell_2(\wh\mD;\U)} := \frac{1}{\#\wh\mD}\sum_{\md\in\wh\mD}
\|\wt u(\md)- \uN(\md;\theta)\|^2_\U 
\ee
over $\theta \in \Theta\subset \R^D$, where $\Theta$ is a set of trainable weights.
This terminology reflects the fact that $\|\wt u- \uN(\theta)\|^2_{\ell_2(\wh\mD;\U)}$ may be regarded as an approximation of the ``ideal risk''
\be
\label{idealoss}
\|u - \uN(\theta)\|^2_{L_2(\mD;\U)} = 
\int_\mD \|u - \uN(\theta)\|^2_\U \,d\mu(\pp) = \E_\mu\big[ \|u- \uN(\theta)\|_\U^2\big], 
\ee
which is indeed the expected error incurred by $\uN(\theta)$ with respect to the
the probability measure $\mu$ on $\mD$. If $\wh\mD$ is a set of samples from
this measure, \eref{datadriv} is a standard Monte Carlo approximation to the expectation.

A rationale behind this approach is therefore to treat the dependence on the spatio-temporal variables deterministically, whereas $\pp$ is viewed as a random variable (where $\mD$ is endowed with a probability measure). Moreover, a mean-square error metric is amenable to standard gradient descent optimization
techniques and leads to more tractable problems in high parametric dimensions (see, for example, \cite{GrVo}) than error measures in terms of $L_\infty$-norms on $\mD$. The norm $\|\cdot\|_\U$ appearing in the loss function typically
involves integration and derivatives. Hence $\|\uN(\theta)\|_\U$ can usually not be evaluated exactly. Here (and in what follows) we ignore the additional perturbations incurred by quadrature which we assume to cause negligible consistency errors.

A resulting approximation $\wt\cS$ to $\cS$, based on a surrogate 
$\uN(\cdot,\md;\theta^*)$, where 
\begin{equation}\label{eq:gitta}
\theta^*\in \argmin_{\theta\in \Theta}\|\wt u- \uN(\theta)\|^2_{\ell_2(\wh\mD;\U)},
\end{equation}
may be viewed as providing a reduced model for $\cM$ provided that the allocated
budget $\Theta$ of degrees of freedom has a feasible size so as to render the
evaluation of $\uN(\cdot,\md;\theta^*)$ significantly more efficient than resorting to 
a highly accurate solver for each instance of $\md$.
 
 Certifying such a reduced model requires quantifying the accuracy of the approximation
 $\wt\cS$. In the light of the preceding discussion, a suitable metric for measuring errors is given by \eref{idealoss}  
with a choice of the probability measure $\mu$ on $\mD$ that may depend on additional 
background information.

In summary, \eref{datadriv} can be viewed as a classial regression problem and one can
tap into results of machine learning to bound the deviation of the empirical loss in 
\eref{datadriv} and the ideal loss in \eref{idealoss},
\be
\label{deviation}
\Bigl|  \|u - \uN(\theta)\|^2_{L_2(\mD;\U)}  - \|\wt u- \uN(\theta)\|^2_{\ell_2(\wh\mD;\U)} \Bigr|
\ee
in expectation or probability, see e.g. \cite{BCDD}.
Corresponding bounds can be based on concentration inequalities (depending on the measure $\mu$)
as well as complexity bounds for the hypothesis class $\cH(\Theta)= \{\uN(\cdot;\theta):
\theta\in \Theta\}$ and the solution manifold $\cM$. We defer more detailed discussions
of this issue to forthcoming work.

Obvious   advantages of this data-driven approach is its conceptual simplicity and that approximations 
naturally respect the  model-compliant metric $\|\cdot\|_\U$. In addition, once one has an accurate
solver at hand, the approach works equally well for linear as wells as nonlinear
models \eref{par}.

On the other hand,
the achievable accuracy depends on $\#\wh\mD$ as well as on the accuracy of the synthetic data $\wt u(\md)$. 
Hence, the fact that this requires computing a potentially large number $\#\wh\mD$ of
training snapshots approximating $u(\md)$,  each one requiring a sufficiently accurate discrete approximate 
solution of \eref{par}, may be seen as a serious disadvantage. A remedy resulting
in a significant complexity reduction is -- very much resembling the construction of reduced bases -- to train the reduced model on residual-type loss functions, an approach that has been termed \emph{physics-informed neural networks} (PINNs) \cite{Kar1,Kar2,Kar3,Mishra1,Mishra2}.

\subsection{Limitations of basic PINN formulations}\label{ssec:PINN}

In its basic version, PINN is based on minimizing an empirical risk involving {\em pointwise}
samples of a residual in space and parameter domain. 
The striking and indeed very tempting
point is that this avoids the need for computing high-fidelity solution snapshots
and  has almost black-box character, enabling the use of generic public domain
software packages for linear as well as nonlinear PDE models. 
In the above terms a typical residual loss function, for $\cF$ as in the stationary case of the general parametric problem \eqref{par}, reads
\be
\label{PINN} 
 \frac{1}{ M N}\sum_{i,j=1}^{M, N} \big|\cF(\uN(x^i,\pp^j;\theta); \pp^j)|^2.
\ee
Such a loss is amenable to standard machine learning strategies based, e.g., on stochastic gradient descent with gradient evaluation by backpropagation.

However, \eqref{PINN} amounts to treating residuals as elements of an $L_2$-space, which is appropriate only under quite specific assumptions on $\cF$. The fundamental issues with such a formulation can be seen in the basic example
\[
\cF(u; \pp)= f + \Div(a(\pp)\nabla u)=0 \quad\mbox{in}\,\,\Omega,\quad u|_{\partial\Omega}=0.
\]
The standard loss function in \eqref{PINN} then reads
\begin{equation}\label{PINNellipt}
    \frac{1}{N}\sum_{j=1}^{N} \biggl\{ \frac{\abs{\Omega}}{M_{\Omega}}  \sum_{i = 1}^{M_{\Omega}} \bigabs{f + \Div\bigl(a(\pp^j) \nabla \uN(x^i, \pp^j;\theta)\bigr)}^2  
    + \frac{\abs{\partial\Omega}}{M_{\partial\Omega}}  \sum_{i = 1}^{M_{\partial\Omega}} \bigabs{\uN(x^i, \pp^j;\theta)}^2 \biggr\},
\end{equation}
where $\pp^j$ are samples in $Y$ and in the two inner summations, $x^i$ are samples in $\Omega$ and $\partial\Omega$, respectively.
As $N , M_{\Omega}, M_{\partial\Omega} \to \infty$, the left part of \eqref{PINNellipt} will tend to 
$\| f + \Div(a\nabla \uN(\theta))\|^2_{L_2(\Omega\times Y)}$, the right part to $\|\uN(\theta)\|_{L_2(\partial\Omega)}^2$. On the one hand, it is well-known that the domain term is too strong and may not even define a bounded expression for non-smooth domains or diffusion coefficients or when $f$ is a proper functional in $H^{-1}(\Omega)$ (such as a trace integral). On the other hand, the boundary term is too weak: a well-posed variational formulation would require an $H^{1/2}(\partial\Omega)$-norm. This mismatch implies that the size of the loss \eqref{PINNellipt} does not provide any certifiable bound for the error $u - \uN(\theta)$ (which here, one would typically aim to measure in $\U= H^1_0(\Omega)$).

\subsection{Proportionality of error and residual}\label{ssec:err-res}

As already emphasized earlier, in view of the uncertainty of optimization success and the lack of any realistic {\em a priori} error bounds for 
DNN approximations of PDEs, a minimal goal regarding accuracy control is to infer
from the current loss, at any stage of the optimization, the size of the
error in the model-compliant metric $\|\cdot\|_\U$. 
That is, any residual-based empirical loss along the lines of \eqref{PINN} or \eqref{PINNellipt} using parametric samples from $\wh\mD$ should remain proportional to $\|u- \uN(\theta)\|^2_{\ell_2(\wh\mD;\U)}$
uniformly in the number of samples $\#\wh\mD$. 
Hence, the remarks in Section \ref{ssec:empiricalrisk}
still apply to relate a residual loss, via $\|u- \uN(\theta)\|^2_{\ell_2(\wh\mD;\U)}$, to $\|u-\uN(\theta)\|^2_{L_2(\mD;\U)}$
in expectation or with high probability.

Our strategy for devising residual loss functionals with this property is based on a weak form of \eqref{par} for each $\pp \in \mD$: viewing the residual as a functional, find $u(\pp) \in \U$ such that
\begin{equation}\label{generalweak}
    \cF(u(\pp); \pp)(v)=0 \quad\text{for all $v \in \V$}    
\end{equation}
with a suitable test space $\V$.
This test space should be chosen (depending on $\U$) such that $\cF(\cdot ; \pp) \colon \U \to \V'$ satisfies
\begin{equation}
\label{res-err}
      \norm{ \cF( w; \pp)}_{\V'} \eqsim \norm{ w - u}_{\U}
\end{equation}
with constants that are uniform in $\pp$. Let us assume for the moment that
the spaces $\U, \V$ can be chosen  {\em independent} of $\pp\in \mD$. 
Then \eref{res-err} implies
\begin{equation}
\label{integrate}
      \norm{ \cF( w ) }_{L_2(\mD;\V')}^2 =  \int_\mD \| \cF(w; \pp) \|^2_{\V'} \,d\mu(\pp) \eqsim  \int_\mD \|u - w\|^2_\U \,d\mu(\pp) = \norm{w - u}_{L_2(\mD; \U)}^2 \,,
\end{equation}
which suggests interpreting $\cF$ as a mapping from $\X = L_2(\mD;\U)$ to $\Y' = (L_2(\mD;\V))' = L_2(\mD;\V')$.

As will be seen later, to guaranty the validity of \eref{generalweak}, it is not always
possible to choose the underlying pair of trial and test space $\U, \V$ independent of
$\pp$. In fact, they may even have to differ as {\em sets} for different $\pp$ which we
refer to as {\em ``essentially different''}. We will show in Section \ref{sssec:lifted}
that under mild assumptions on   $\pp$-dependence integration in \eref{integrate}
is well-defined and one still  arrives at an error-residual relation of a 
{\em single variational formulation} over a pair of suitably extended 
versions of the spaces $\X, \Y$, defined above.

At any rate, the corresponding loss functional in the spirit of \eqref{PINN} thus takes the form
\[ 
   \frac{1}{N} \sum_{i = 1}^{N}  \| \cF(\uN(x^i,\pp^j;\theta); \pp^j) \|^2_{\V'} .
\]
Unless \eref{res-err} holds when $\V$ itself is also an $L_2$-space, the obvious difficulty in evaluating  the dual norms 
\[  
	 \|\cF(w)\|_{\V'}=\displaystyle\max_{v\in\V}\frac{\cF(w)(v)}{\|v\|_\V}
\] 
is the required maximization over the entire Hilbert space $\V$.

\subsection{Relation to existing results}

Whenever the solution of a PDE is characterized by an energy minimization principle
(the elliptic case) a natural way of producing approximate solutions in the right metric
is to minimize the corresponding energy functional over the chosen hypothesis class.
In a finite element context this can actually be used to produce also a reliable and efficient posteriori
error bounds. These concepts do unfortunately not carry over to more general hypothesis classes like DNNs so that this approach does in general not come with accuracy quantification. This remains crucial because the minimization of such quadratic functionals over highly nonlinear hypothesis classes remains problematic.
Besides, this approach is limited to elliptic problems.   

The idea of minimizing instead a variationally correct residual   applies to a wider scope and has been considered  in several previous works, sometimes referred to as WAN (weak adversarial networks). More specifically,
the authors in \cite{Friedrichs,BYZZ,ZBYZ} discuss specific PDE models namely so called
{Friedrichs systems} as well as second order elliptic problems, primarily not for
the purpose of model reduction or operator learning but to devise DNN-based PDE solvers.
This has been the primary focus (albeit not always the original motivation) also in several more recent works such as \cite{Canuto,Bertol,Pardo,Urban}, where the considered PDEs are posed on domains of small dimension. In all these works, the need to minimize a dual norm has been identified as a major theoretical as well as practical hurdle. For instance, the natural idea of solving $\min_{\theta\in\Theta}\|\cF(\uN(\theta))\|_{\Y'}$ by repeatedly performing 
some gradient ascent steps on 
\begin{equation}\label{Fminmax}  \frac{\cF(\uN(\theta))(\uN'(\psi))}{\|\uN'(\psi)\|_\Y} 
\end{equation}
over elements $\psi\in\Psi$ of a suitable, sufficiently rich test system $\Psi$, followed by some descent steps over the primal system $\Theta$, is problematic and  not very reliable. Again, ensuring convergence of the optimization over the test space is an issue. For low-dimensional problems, some of these works consider therefore more manageable test systems like wavelet systems \cite{Urban}, finite elements \cite{Canuto} or trigonometric systems \cite{Pardo}.
These remedies are unfortunately not feasible in the present high-dimensional parametric setting. Finally, enforcing essential boundary conditions in the correct norms for DNN approximations poses some challenges \cite{Canuto}.

\subsection{Hybrid approximation format}\label{ssec:hybrid}

Due to the issues mentioned above, rather than approximating parameter-dependent solutions directly by DNNs in all variables, we focus here on a hybrid approximation format for that uses piecewise polynomial approximations in the spatio-temporal variables combined with DNN approximations of the parameter-dependent coefficients. 
In doing so we exploit finite element methods in small spatial dimensions, where concepts for the reliable evaluation of dual norms are available. 

While the specific choice of finite element basis functions depends on the PDE model at hand, the general structure can be described as follows.
Given finite element basis functions $\{\phi_{i}\}_{i \in \mathcal{I}}$ on the spatio-temporal domain under consideration, we use approximations of the form
\be
\label{eq:hybridformat}
u(t,x, \pp) \approx \sum_{i \in \mathcal{I}} \mathbf{u}_{\mathrm{NN},i}(\pp; \theta)\phi_{i}(t,x),
\ee
where $ \pp \mapsto \mathbf{u}_\mathrm{NN}(\pp; \theta) = \bigl( \mathbf{u}_{\mathrm{NN},i}(\pp; \theta)\bigr)_{i \in\mathcal{I}}$ is realized by a neural networks with trainable parameters $\theta$. Representations of this form have been tested for parametric elliptic problems, for example, in \cite{GPRSK:21,CGE:23}; differently from our approach, these approximations were trained on samples of finite element solutions as in \eqref{eq:gitta}, where the loss function quantifies the relative error in energy norm to the parameter-dependent discrete solutions.

Let us note that instead in our approach, by the use of appropriate residual-based loss functions, we obtain estimates of the error (up to data oscillation) with respect to the \emph{exact solution} of the underlying continuous problem.

The format \eqref{eq:hybridformat} is instrumental for treating dual norms in variationally correct loss functions without the use of min-max optimization as in \eqref{Fminmax}.
It also offers a number of further significant advantages. First,
rigorously enforcing essential boundary conditions for DNNs is an issue, see \cite{Canuto} for various strategies. 
The hybrid format instead greatly facilitates incorporating essential (spatial) boundary conditions.    Second, since the input parameters of the coefficient
neural networks are just the model parameters, training requires only efficient backpropagation.
When using DNNs with both spatio-temporal and parametric variables 
as input parameters, one needs to compute higher-order derivatives with respect to all variables. This has to be
done with great care to avoid a significant loss of efficiency, see e.g. \cite{Pardo}.

\subsection{Novelty and organization of the material}

The central contribution in the present work is to propose and analyze concepts 
for operator learning for a wide and diverse scope of PDE models with rigorous accuracy control in model-compliant metrics that are feasible in high-dimensional parametric regimes. The main conceptual constituents
can be summarized as follows:
\begin{enumerate}
\item 
Reformulate the given family of parametric PDEs as a {\em single variational problem} posed
over Hilbert spaces of functions of spatio-temporal and parametric 
variables, see Section \ref{sec:residual}. Roughly, this entails treating spatio-temporal variables deterministically while a probabilistic interpretation of the parametric domain allows us to resort to Monte Carlo sampling as a basis for a learning approach.
\item 
Establish {\em well-posedness} of the single high dimensional problem (in the sense of the Babu\v{s}ka-Ne\v{c}as Theorem) from judiciously chosen {\em stable variational formulations} of the parametric (low-dimensional)
fiber problems, see Theorem \ref{thm:suff}. This gives rise to an {\em ideal residual
loss function} which is indeed variationally correct. The term ``ideal'' reflects that this a preliminary step because such residual loss functions typically involve 
a {\em dual norm} whose computational evaluation is problematic. Hence, this serves 
as the starting point for deriving from such ideal losses ones that are still variationally correct but are also practically feasible. We stress that the scope of PDE models to which our approach applies corresponds in essence to the ability of contriving stable variational formulations.
\item  Section \ref{sec:3} is then devoted to \emph{practical variationally correct residual loss functions}. Specifically, this is based on first reformulating, if necessary,
a given parametric PDE as {\em first-order systems}. We then distinguish in Section \ref{ssec:3.1}   two 
scenarios given in \eqref{eq:sf} and \eqref{eq:uwf}, respectively: In the first scenario (\ref{eq:sf}), we identify stable variational formulations inducing 
an operator that maps (an appropriate subspace of) its graph space as trial space bijectively onto an $L_2$-type space. In the second scenario (\ref{eq:uwf}), an $L_2$-type space is taken as trial space
and the concept of an \emph{optimal test norm} is used to define a test space leading to an induced operator with condition number equal to one. This is referred to as 
{\em ultra-weak formulation}.

The relevance of (\ref{eq:sf}) is clear. Since the residual is measured in $L_2$, the ideal residual loss is also practical. We demonstrate this variant in Sections \ref{ssec:3.2}, \ref{ssec:3.3} for a classical second order elliptic problem with
parameter-dependent diffusion coefficient and an analogous parabolic initial-boundary value problem with a space-time variational formulation. In Section \ref{ssec:ellnumer}
we present some numerical experiments for the elliptic example. 

In (\ref{eq:uwf}) we still have to deal with a non-trivial dual norm in the ideal loss function.
The point is that this ultra-weak formulation conveniently accommodates stable {\em discontinuous Petrov-Galerkin} (DPG) formulations, see Section \ref{ssec:dpg}. 
It is the combination of DPG concepts with the hybrid format of our hypothesis classes, as described in Section \ref{ssec:hybrid},
that allows us to derive practical variationally correct loss functions. This is exemplified for linear transport equations with parameter dependent convection field
in Section \ref{sssec:transport}. Corresponding numerical experiments are presented in
Section \ref{sssec:transportnum}.
\item  The hybrid format mentioned above is essential for scenario (\ref{eq:uwf}). Roughly,
the elements in corresponding hypothesis classes are piecewise polynomials as 
functions of spatio-temporal variables with parameter-dependent coefficients that may, for instance, be represented as DNNs. This is not strictly needed for scenario (\ref{eq:sf}). We have nevertheless tested it in comparison to a standard PINN approach for (\ref{eq:sf}) as well. 
Moreover, we discuss sparsification formats that do reduce representation complexity
significantly. Aside from this,
  it offers several major advantages: it facilitates naturally enforcing essential boundary conditions, and
it does not require expensive forward differentiation with respect to the spatio-temporal variables in the training process. In the absence of parameters, one falls back 
on well-established and accurate PDE solvers.
\end{enumerate}

A few words on the choice of test cases are in order. A major distinction between the elliptic case and transport equations lies in the fact in the former case trial and test spaces can be chosen as {\em independent} of the parameters. In this case graph spaces as trial spaces
are appropriate, which is scenario (\ref{eq:sf}).  By contrast, for transport equations either trial or test space
depends {\em essentially} on the parameters, that is, the spaces vary \emph{as sets} when the parameters vary. $L_2$ is then a space that robustly accommodates the solutions for all permissible parameters.
 For this reason, we resort to scenario (\ref{eq:uwf}) in this case. As different as (\ref{eq:sf}) and (\ref{eq:uwf}) appear to be, they are nevertheless closely related as argued in Section \ref{ssec:3.1}. Moreover, when finally resorting to DPG
concepts, one relies on elements of both scenarios, because the additional auxiliary 
skeleton variables can be interpreted as approximations in the graph space (see also
the proof of well-posedness in \cite{BDS}).

\section{Variational Formulations and Residual Loss Functions}\label{sec:residual}

\subsection{Parametric fiber problems}\label{sssec:fiber}
As indicated earlier, the validity of \eref{res-err}
hinges on a suitable variational formulation for each $\pp\in \mD$.
For linear problems, \eref{generalweak} takes for each $\pp\in\mD$ the form: find $u(\pp) \in \U_\pp$ such that 
\begin{equation}
\label{eq:paramdepform}
   b(u(\pp), v; \pp) - f(v) = 0\qquad \text{ for all $v\in \V_\pp$,}	
\end{equation}
viewed as a family of linear operator equations $\cB_\pp u(\pp)= f$ in $\U_\pp$, where $(\cB_p w)(v)
= b(w,v;\pp)$, $v\in \V_p$. As mentioned earlier, the  spaces $\U_\pp$ and $\V_\pp$
are to be chosen so as to ensure well-posedness of \eref{eq:paramdepform}. 
By the celebrated Babu\v{s}ka-Ne\v{c}as Theorem (see e.g.\cite{Braess}), this is equivalent to the validity of
\begin{equation}
\label{inf-sup-pp}
\begin{aligned}
\sup_{w\in\U_\pp}\sup_{v\in\V_\pp}\frac{b(w,v;\pp)}{\norm{w}_{\U_\pp} \norm{v}_{\V_\pp}}\le C_b,\quad
 \inf_{w\in\U_\pp}\sup_{v\in\V_\pp} \frac{b( w,v; \pp)}{\norm{w}_{\U_\pp} \norm{v}_{\V_\pp}}\ge c_b,\\
\text{for each $v\in \V_\pp\,\,\exists\, w_v\in \U_\pp$ such that $b( w_v,v ; \pp) \neq 0$,}
\end{aligned}
\end{equation}
for some constants $0<c_b=c_b(\pp), C_b=C_b(\pp)<\infty$, that generally may depend on $\pp$.  We  remark that respective stable variational formulations are known
for a wide scope of problems covering dissipative, indefinite, as well as dispersive models; see, for example, \cite{CDG}.

\begin{rem}
Given a (linear) PDE $\cB_\pp u=f$, underlying \eref{eq:paramdepform}, in strong form, it should be viewed as part of the problem
to identify a suitable pair of trial and test space $\U_\pp, \V_\pp$ (first on the infinite dimensional level) for which \eref{inf-sup-pp} holds. If the problem is {\em elliptic}
(or coercive) this choice is easy, namely $\U_\pp=\V_\pp$, where $\U_\pp$ is the energy space.
In general, when the problem is indefinite, non symmetric, or singularly perturbed,
one may have to accept $\U_\pp\neq \V_\pp$ in order to warrant stability,
see Proposition \ref{rem:uw} and Remark \ref{rem:roadmap} for a ``roadmap'' of how to proceed.
\end{rem}

As explained in Section \ref{ssec:err-res}, if all the spaces $\U_\pp$, $\V_\pp$ agree as sets $\U,\V$, with respective 
equivalent norms $\|\cdot\|_\U$, $\|\cdot\|_\V$, the Bochner spaces $\X=L_2(\mD;\U)$, $\Y = L_2(\mD;\V)$ provide a natural pair of trial and test spaces for a single variational 
problem determining the fiber solutions of \eref{eq:paramdepform} as functions of spatial and parametric variables. We refer to this as the \emph{lifted problem}.

Unfortunately, important examples that are discussed below reveal that to ensure
well-posedness in the sense of \eref{inf-sup-pp}, it may be necessary to accept
families of trial and test pairs $\U_\pp$, $\V_\pp$ that differ for different $\pp$ in an {\em essential way}, by which we mean that they differ as sets with non-equivalent norms.

To be still able to recast the family \eref{eq:paramdepform} of weak formulations as a {\em single
  space-parameter} weak formulation, we need to address two issues: first, what are 
 suitable replacements of the classical Bochner spaces as trial and test spaces and second, how does well-posedness of fiber problems \eref{eq:paramdepform} relate to
 the well-posedness of the resulting lifted problem?
 
  Regarding the first issue, the remarks in   Section \ref{ssec:empiricalrisk} suggest choosing $\X$ as a suitable space of functions $v \colon \pp\mapsto v(\pp) \in \U_\pp$ such that the quantities 
\[
	\int_\mD \norm{v(\pp)}_{\U_\pp}^2\,d\mu(\pp) < \infty ,
\]
and similarly for $\Y$ with $\U_\pp$ replaced by $\V_\pp$, remain well-defined.
To ensure that resulting notions of spaces $\X$ and $\Y$ are meaningful, we need some mild assumptions on the dependence of the fiber spaces $\U_\pp$ and $\V_\pp$ on the parameter $\pp$ that warrant measurability of elements in the Cartesian products
$\prod_{\pp\in\mD}\U_\pp$, $\prod_{\pp\in\mD}\V_\pp$. This leads us to the notion of
\emph{direct integrals} discussed next.

\subsection{Direct integrals}\label{sssec:Hilbert}

We now consider a general notion of measurability, going back to von Neumann \cite{JvN}, of mappings $\pp \mapsto v(\pp) \in \U_\pp$ as discussed above. The resulting Hilbert spaces can be regarded as generalizations of the direct sums of Hilbert spaces to the case of uncountable index sets. 

\begin{definition}\label{def:meassect}
	Let $(\W_\pp)_{\pp\in\mD}$   be a family of separable Hilbert spaces.
 We then call a sequence $(\xi_n)_{n\in\N}$ in the Cartesian product $\prod_{\pp \in \mD} \W_\pp$ a \emph{fundamental sequence of $\mu$-measurable sections} if for all $n$ and $m$, the function $\mD \ni \pp \mapsto \langle \xi_n(\pp), \xi_m(\pp)\rangle_{\W_\pp}$ is $\mu$-measurable, and for each $\pp \in \mD$, the finite linear combinations of $\xi_n(\pp)$, $n\in\N$, are dense in $\W_\pp$. Under these conditions, we say that $v \in \prod_{\pp \in \mD} \W_\pp$ is $\mu$-measurable if $\pp\mapsto \langle v(\pp), \xi_n(p)\rangle_{\W_\pp}$ is $\mu$-measurable for %
 $n\in\N$.
\end{definition}

The following result is shown in \cite[Ch.~IV.8]{Takesaki} (see also \cite[Ch.~VIII.3]{DautrayLions}).

\begin{theorem}
	\label{thm:mes}
		Let $(\W_\pp)_{\pp \in \mD}$ be a family of separable Hilbert spaces 
		such that $\prod_{\pp\in \mD}\W_\pp$ contains a
		fundamental sequence of $\mu$-measurable sections.
		Then the spaces $L_2\bigl(\mD, (\W_\pp)_{\pp \in \mD} \bigr)$ defined by
\[
L_2\bigl(\mD, (\W_\pp)_{\pp \in \mD} \bigr)  = \biggl\{ v \in \prod_{\pp\in\mD} \W_\pp \text{ $\mu$-measurable} \colon \int_{\pp \in \mD} \norm{v(\pp)}_{\W_\pp}^2\,d\mu(\pp)<\infty  \biggr\},  	
		\]
		where elements that agree $\mu$-almost everywhere are identified, endowed with the inner product
	\[
		\langle v, w\rangle_{L_2(\mD, (\W_\pp)_{\pp \in \mD})} = \int_\mD \langle v(\pp),w(\pp)\rangle_{\W_\pp} \,d\mu(\pp),
	\]
	 are Hilbert spaces.
	\end{theorem}
	The spaces provided by Theorem \ref{thm:mes} are called \emph{direct integrals} (or also \emph{Hilbert integrals}).
The simplest scenario where Theorem \ref{thm:mes} applies concerns the case where
$\W_\pp= \W$, with equivalent norms, uniformly in $\pp\in\mD$.
\begin{rem}
\label{rem:justified}
If $\W$ is a separable Hilbert space, the Bochner space $\X = L_2(\mD, \W)$ fulfills Assumption \ref{ass:HilbertIntegral}, see \cite[VIII.3, Rem.~2]{DautrayLions}. This justifies the previous discussion leading to \eref{integrate}.
\end{rem}

We will also need an analogous characterization of the dual   
of $L_2\bigl(\mD;,(\W_\pp)_{\pp \in \mD} \bigr)$, which will be applied, in particular, to the space  $\bigl(L_2(\mD,(\V_\pp)_{\pp\in\mD})\bigr)'$. This is instrumental for arriving at a well-defined ideal residual loss function also in scenarios where the fiber test spaces $\V_\pp$ differ essentially. 

In general, for a Hilbert space $\H$ with inner product $\lll\cdot,\cdot\rr_\H$, we denote by $\cR = \cR_\H \colon \H' \to \H$ the inverse of the Riesz isometry identifying $\H$ with its dual $\H'$, which for $\ell \in \H'$ is given by
\be
\label{Riesz}
\lll \cR \ell,v\rr_{\H} = \ell(v),\quad v\in \H.
\ee
For direct integrals $\H=\big(L_2(\mD,(\W_\pp)_{\pp\in\mD})\big)$ this mapping can be characterized in terms of the fiber Riesz lifts  $\mathcal{R}_\pp \colon \W_\pp' \to \W_\pp$ defined for each $\pp \in \mD$ and $\ell_\pp \in \W'_\pp$ by  $\langle \mathcal{R}_\pp \ell_\pp, v \rangle_{\W_\pp} = \ell_\pp(v)$ for each $v \in \W_\pp$.

\begin{theorem}
\label{lem:dualHilbertIntegral}
 For $\H= \big(L_2(\mD,(\W_\pp)_{\pp\in\mD})\big)$ we have 
 \begin{equation}\label{eq:rieszdualiso}
	\H' =  \left\{  \ell \in \prod_{\pp\in\mD} \W_\pp' \text{ $\mu$-measurable} \colon \int_{\pp \in \mD} \norm{\ell(\pp)}_{\W_\pp'}^2\,d\mu(\pp)<\infty  \right\}
 \end{equation}
	with inner product
	\[
	   \langle \ell, \tilde\ell\rangle_{\H'} = \int_\mD \langle \ell(\pp),\tilde\ell(\pp)\rangle_{\W_\pp'} \,d\mu(\pp),	 \quad \ell,\tilde\ell \in \H',
	\]
	and for each $\ell \in \H'$, 
	\begin{equation}\label{eq:rieszsum}
		 \langle \cR\ell, v \rangle_\H 
		  = \int_\mD \langle \cR_\pp \ell(\pp), v(\pp)\rangle_{\W_\pp} \,d\mu(\pp),
	\end{equation}
which says that for each $\ell\in \H'$
\be
\label{exchange}
\ell = \prod_{\pp\in\mD}\ell(\pp)\quad\text{such that}\quad (\cR\ell)(\pp)= \cR_\pp\ell(\pp),\quad \pp\in\mD.
\ee
\end{theorem}

\begin{proof}
	Let $\ell\in \H'$. Then for all $v\in \H$,  
\begin{multline*}
\ell(v)= \lll \cR \ell,v\rr_\H= \int_\mD \langle (\cR \ell)(\pp),v(\pp)\rangle_{\W_\pp}{d}\mu(\pp) \\ = \int_\mD \cR_\pp^{-1}(\cR \ell)(\pp)(v(\pp)) {d}\mu(\pp)
= \int_\mD  \ell_\pp(v(\pp)){d}\mu(\pp).
\end{multline*}
Defining $\ell_\pp := \cR_\pp^{-1}(\cR \ell)(\pp)\in \W_\pp'$ we therefore have
\begin{equation}
\label{eq:rieszsum1}
\langle \mathcal R \ell, v\rangle_\H 
= \int_\mD \langle \mathcal R_\pp \ell_\pp,v(\pp)\rangle_{\W_\pp}{d}\mu(\pp).\end{equation}
Given now a fundamental sequence of $\mu$-measurable sections $\phi_n(\pp)\in\W_\pp$,
let $\xi_n(\pp):= \cR_\pp^{-1}\phi_n(\pp)$, $n\in\N$. Since	$\cR_\pp$ are isometries
the $\xi_n(\pp)$ are dense in $\W_\pp'$. Noticing that
$\pp\mapsto \lll \xi_n(\pp),\xi_m(\pp)\rr_{\W_\pp'}= \lll \phi_n(\pp),\phi_m(\pp)\rr_{\W_\pp}$, it follows that the $\xi_n(\pp)$ form a fundamental system in $\W_\pp$, $\pp\in\mD$. Moreover,
$$
\lll \ell_\pp,\xi_n(\pp)\rr_{\W_\pp'}= \lll \cR_\pp^{-1}(\cR \ell)(\pp),\xi_n(\pp) \rr_{\W_\pp'}= \cR_\pp^{-1}(\cR \ell)(\pp)(\phi_n(\pp))=\lll 
(\cR \ell)(\pp),\phi_n(\pp)\rr_{\W_\pp}
$$
shows measurability of $\pp\mapsto  \lll \ell_\pp,\xi_n(\pp)\rr_{\W_\pp'}$.
 Thus $\ell$ can be identified with a $\mu$-measurable function $\ell\in\prod_{\pp\in\mD}\W'_\pp$ with $\ell(\pp)=\ell_\pp$.
In particular, \eqref{eq:rieszsum} follows from \eqref{eq:rieszsum1}. 
	Furthermore, we observe that $\|\ell(\pp)\|_{\W'_\pp} = \|(\cR\ell)(p)\|_{\W_\pp}$ 
	and   
$$
\int_\mD\|\ell(\pp)\|^2_{\W'_\pp}{d}\mu(\pp)  = \int_\mD\|(\cR\ell)(\pp)\|^2_{\W_\pp}{d}\mu(\pp)<\infty,$$
	and we thus obtain \eqref{eq:rieszdualiso}. %
\end{proof}

An application of these concepts to the families $(\U_\pp)_{\pp\in\mD}$, $(\V_\pp)_{\pp\in\mD}$ in \eref{eq:paramdepform}, where these spaces may differ essentially from each other,
  requires verifying the validity of the following

\begin{property}\label{ass:HilbertIntegral}
	For the families $(\U_p)_{p \in P}$,  $(\V_p)_{p \in P}$ of separable Hilbert spaces,   there exist fundamental sequences of $\mu$-measurable sections $(\phi_n)_{n\in\N}$ and $(\psi_n)_{n\in\N}$, respectively.
\end{property}
A convenient criterion that   applies in our context later below, can be formulated as follows.
\begin{prop}
\label{prop:cap}
	If $\bigcap_{\pp\in\mD}\U_\pp$
and  $\bigcap_{\pp\in\mD}\V_\pp$ are dense in each $\U_\pp, \V_\pp$, respectively, then Property \ref{ass:HilbertIntegral} is satisfied.
\end{prop}
This observation is of interest also in the following   respect. The above notion of direct integral 
is in some sense very flexible and permits counterintuitive facts. For instance, if
 some family of Hilbert spaces $\W_\pp$, $\pp \in \mD$, gives rise to a direct integral   $L_2(\mD,(\W_\pp)_{\pp\in\mD})$, not every choice of closed subspaces $\V_\pp\subset\W_\pp$ induces a direct integral structure $L_2(\mD,(\V_\pp)_{\pp\in\mD}) \subset L_2(\mD,(\W_\pp)_{\pp\in\mD})$. We refer the interested reader to Example 
 \ref{ex:B} in Appendix B.  
 
However, under the assumption in Proposition \ref{prop:cap}, the expected subspace relation holds. The precise circumstances are given in Proposition \ref{prop:subspace}
that can also be found in Appendix B.

\subsection{The lifted problem and its well-posedness}\label{sssec:lifted}

We are now ready to reformulate \eref{eq:paramdepform} as a single linear operator equation
\be
\label{lB}
\cB u = F,
\ee
involving functions of spatial and parametric variables, so as to ensure validity of a
tight error-residual relation.

In general, for such a linear operator equation, the question for which data $F$ \eref{lB} has a unique solution $u$ in a suitable trial space $\X$,
 depends on the {\em mapping properties}
of $\cB$, viewed as a mapping from some (infinite-dimensional) {\em trial space} $\X$ onto a suitable target space, accommodating the data $F$. If 
Property \ref{ass:HilbertIntegral} holds, which we assume throughout the remainder of this section, natural candidates for a reformulation of \eref{eq:paramdepform} are
the spaces %
\begin{equation}\label{XY}
	\X = L_2\bigl(\mD, (\U_\pp)_{\pp \in \mD} \bigr), 
	\qquad \Y = L_2\bigl(\mD, (\V_\pp)_{\pp \in \mD} \bigr)
\end{equation}
with respective fundamental sequences $(\phi_n)_{n\in\N},(\psi_n)_{n\in\N}$.
In addition, we use the following property that needs to be verified for the given problem.

\begin{property}
\label{ass:bil}
	Assuming Property \ref{ass:HilbertIntegral}, for each $\pp \in \mD$, let $b(\cdot,\cdot;\pp)\colon \U_\pp \times \V_\pp \to \R$ be a bounded bilinear form such that for all $n, m \in \N$, the mapping $
	   \mD \ni \pp \mapsto 	b(\phi_n(\pp), \psi_m(\pp); \pp)
	$
	is measurable.
\end{property}
Under these conditions, the quantities  
\begin{equation}\label{eq:paramweak}
  b(u,v):= 
    \int_\mD b\bigl(u(\pp), v(\pp); \pp \bigr)\,d\mu(\pp),
	\qquad F(v) = \int_\mD f\bigl(v(\pp)\bigr)\,d\mu(\pp),
\end{equation}
are well-defined. Solving \eref{lB} for $\cB$, defined by $(\cB w)(v)= b(w,v)$, $w\in \X$, $v\in \Y$,
 boils down  to finding $u\in \X$ such that for $F\in \Y'$ 
 \be
\label{wlB}
b(  u,v) = F(v),\quad \forall\, v\in \Y.
\ee
As has been used already for the fiber problems, well-posedness of \eref{wlB} is 
{\em characterized} by the Babu\v{s}ka-Ne\v{c}as Theorem through the existence
of constants
$0<c_b\le C_b<\infty$ such that
\be
\label{inf-sup}
\begin{array}{c}
\displaystyle\sup_{w\in\X}\sup_{v\in\Y}\frac{b( w,v)}{\|w\|_\X \|v\|_\Y}\le C_b,\quad
 \inf_{w\in\X}\sup_{v\in\Y} \frac{b( w,v)}{\|w\|_\X \|v\|_\Y}\ge c_b,\\[18pt]
\text{for each $v\in \Y\,\,\exists\, w_v\in \X$ such that $b( w_v,v) \neq 0$.}
\end{array}
 \ee
 \begin{rem}\label{rem:dualinfsup}
 An equivalent set of condition is obtained when replacing the surjectivity condition in the second line by an inf-sup condition with the roles of  $\X$ and $\Y$ being interchanged, since this
 is equivalent to the dual $\cB':\Y\to \X'$ being an isomorphism.
 \end{rem}

Here we emphasize the equivalence of \eref{inf-sup} to the fact that $\cB:\X\to\Y'$
is an {\em isomorphism} which in turn means that
\be
\label{err-res}
c_b\|u- w\|_\X\le \|\cB w - F \|_{\Y'}\le C_b\|u- w\|_\X, \quad \forall \, w\in \X,
\ee
Hence $\cB:\X\to\Y'$ has a bounded {\em condition number}
\be
\label{condition}
\kappa_{\X,\Y}(\cB) := \|\cB\|_{\X\to\Y'}\|\cB^{-1}\|_{\Y'\to\X}\le \frac{C_b}{c_b},
\ee
so that the smaller the ratio $C_b/c_b$ of continuity and inf-sup constant the tighter
is the relation of an error in $\X$ to the residual in $\Y'$.
We refer to Section \ref{ssec:3.1} for corresponding strategies for identifying
suitable pairs of trial and test spaces, including cases where  $\frac{C_b}{c_b}=1$, that is, the variational problem has an ideal condition.

We are now ready to clarify how well-posedness of the fiber problems \eref{eq:paramdepform} relates to well-posedness of the lifted problem \eref{wlB}. 
\begin{theorem}
\label{thm:suff}
If Properties \ref{ass:HilbertIntegral} and \ref{ass:bil} hold,
the operator $\cB\colon \X\to \Y'$ defined in \eqref{eq:paramweak}, with $\X,\Y$ as in \eqref{XY}, is an isomorphism if and only if there exist uniform constants $c_b, C_b>0$ such that for $\mu$-almost all $\md \in \mD$, the conditions \eref{inf-sup-pp} hold.
In this case $\cB$ satisfies \eqref{inf-sup} with the same constants $c_b,C_b$.
\end{theorem} 

\begin{proof}
First, let \eqref{inf-sup-pp} hold for $\mu$-almost all $\pp\in \mD$. 
This is equivalent to $\cB_\pp \colon \U_\pp \to \V_\pp'$ defined by 
$\cB_\pp w = b(w, \cdot; \pp) \in \V_\pp'$ being an isomorphism for $\mu$-almost all $\pp \in \mD$ with 
\[ c_b \norm{ w }_{\U_\pp} \leq \norm{\cB_\pp w}_{\V_\pp'} = \sup_{0\neq v \in \V_\pp} \frac{b(w, v; \pp)}{\norm{v}_{\V_\pp}} \leq C_b \norm{w}_{\U_\pp} \quad \text{for all $w \in \U_\pp$.}\] 
Moreover, for $v \in \Y$,
\[
   b(w, v)  \leq \int_\mD C_b \norm{w(\pp)}_{\U_\pp}\norm{v(\pp)}_{\V_\pp}  \,d\mu(\pp) \leq C_b \norm{w}_\X \norm{v}_\Y,
\]
confirming continuity of $b(\cdot,\cdot)$.
Regarding the inf-sup condition, we note that by boundedness and measurability of $b(\cdot,\cdot)$, we have for $w\in\X$ that  $$\cB w = b(w,\cdot)\in\Y'$$ and thus
 we have, by Lemma \ref{lem:dualHilbertIntegral}, 
 $$\langle \mathcal R\cB w,v\rangle =\int_\mD b(w(\pp),v(\pp),\pp){d}\mu(\pp) = \int_\mD \langle \mathcal R_\pp\cB_\pp w(\pp),v(\pp)\rangle_{\V_\pp}{d}\mu(\pp).
  $$
Therefore,   $v = \mathcal R \cB w$ satisfies
 \[
\begin{aligned}
	b( w, v) &= \int_\mD b(w(\pp),  \mathcal{R}_\pp\cB_\pp w(\pp); p)\,d\mu(\pp)  
	 = \int_\mD \norm{ \cB_\pp w(\pp)}_{\V_\pp'}^2\,d\mu(\pp) 	\\
	& \geq c_b \norm{w}_\X \biggl( \int_\mD \norm{ v(\pp)}_{\V_\pp}^2\,d\mu(\pp) \biggr)^{1/2} .
\end{aligned}
\]
The third condition in \eqref{inf-sup} then follows in the same manner by Remark \ref{rem:dualinfsup}, using that also $\cB_\pp'$ is an isomorphism for $\mu$-almost all $\pp\in\mD$.

To prove the converse
	assume now that there are constants $c_b,C_b$ such that \eqref{inf-sup} holds but there are no uniform lower and upper bounds for $c_b(\pp), C_b(\pp)$ in \eqref{inf-sup-pp}. 
	Observe that as $\hat w_b = \cB  w=b(w,\cdot)\in\Y'$ for all $w\in\X$ we have measurability of 
	$$\pp\mapsto \|\hat w_b(\pp)\|_{\V'_\pp} = \sup_{v\in\V_\pp}\frac{b( w(p),v;\pp)}{\|v\|_{\V_\pp}}= \|(\cB w)(\pp)\|_{\V_\pp'} .
	$$
	Consider the set $P_{c_b,w} = \{ \pp\in\mD: \|\hat w_b(p)\|_{\V'_\pp}<\frac{c_b}{2}\|w(\pp)\|_{\U_\pp}  \}. $ If $\mu(P_{c_b,w})=0$ for all $w\in\X$ there is a uniform lower bound. Else take $w\in\X$ such that $\mu(P_{c_b,w})>0$. Then choose 
$$
	v_w = \mathcal R (\hat w_b\mathbf{1}_{P_{c_b,w}})  \in\Y,
$$
	where $\mathbf{1}_{P_{c_b,w}}$ is the indicator function of ${P_{c_b,w}}$.
Then we have, by definition of  $P_{c_b,w}$,
$$
	b(w\mathbf{1}_{P_{c_b,w}},v_w) =\|\hat{w}_b\mathbf{1}_{P_{c_b,w}}\|_{\Y'}
	< \frac{c_b}2 \|w\mathbf{1}_{P_{c_b,w}}\|_\X 
	\|v_w\|_\Y.
$$
Since $w\in\X$ implies $w\mathbf{1}_{P_{c_b,w}}\in\X$ and 
\[ \frac{b(w\mathbf{1}_{P_{c_b,w}},v_w)}{\|v_w\|_\Y}= \sup_{v\in\Y}\frac{b(w\mathbf{1}_{P_{c_b,w}},v )}{\|v\|_\Y}  , \]
 we arrive at a contradiction to the inf-sup condition  \eqref{inf-sup} for $b(\cdot,\cdot)$.

Next we   use Remark \ref{rem:dualinfsup} and note that interchanging the roles of $w$ and $v$, the same reasoning yields the contradiction 
$$
	\inf_{v\in\Y}\sup_{w\in\X}\frac{b(w,v)}{\|w\|_\X\|v\|_\Y}<  \frac{c_b}{2}.
$$
	Finally, defining the set 
$$
Q_{C_b,w} = \{ \pp\in\mD: \|\hat w_b(p)\|_{\V'_\pp}> 2 C_b\|w(p)\|_{\U_\pp}  \}. 
$$
we derive in the same fashion for any $w\in\X$ such that $\mu({Q_{C_b,w}})>0$,	that 
$$  
b(w\mathbf{1}_{Q_{C_b,w}},v_w) =b(w\mathbf{1}_{Q_{C_b,w}},\mathcal R (\hat w_b \mathbf{1}_{Q_{C_b,w}}))> 2C_b\|w\mathbf{1}_{Q_{C_b,w}}\|_\X\|v_w\|_\Y,
$$
finishing the proof.	
\end{proof}

Recall that the integration over $\mD$ is replaced in computations by a weighted summation
over discrete subset $\wh\mD\subset \mD$. Introducing corresponding analogous objects
$\wh b(\cdot,\cdot), \wh\cB, \wh F$ in \eref{eq:paramweak}, this suggests
  defining in analogy to \eref{XY} the {\em parameter-discrete}
direct sum Hilbert spaces
\be
\label{XYhat}
\wh\X =  \ell_2\bigl(\wh\mD;(\U_\md)_{\pp \in \wh\mD}\bigr), \quad \wh\Y = \ell_2\bigl(\wh\mD;(\V_\md)_{\pp \in \wh\mD}\bigr).
\ee
We record the following trivial fact.
\begin{rem}
\label{rem:wh}
When the fiber problems \eref{eq:paramdepform} satisfy \eref{inf-sup-pp} uniformly in $\pp$
 the problem $\wh b(u,v)=\wh F(v)$
for all $v\in \wh\Y$ is well-posed for any subset $\wh\mD$ of $\mD$ as well.
\end{rem}

In summary, once \eref{inf-sup} has been established, the continuous and parameter-discrete error-residual relations
\be
\label{upshot}
\|u- w\|_\X\eqsim \|F-\cB w\|_{\Y'}, \quad \|u- w\|_{\wh\X}\eqsim \|F-\cB w\|_{\wh\Y'},
\quad w\in \X,
\ee
suggest viewing $\|F-\cB w\|_{\Y'}$, $\|F-\cB w\|_{\wh\Y'}$ as {\em ideal} loss functions.
We proceed discussing next the numerical evaluation and approximation of such ideal loss functions.

\subsection{Obstructions and main goal}\label{sssec:Riesz}

In a classical (low-dimensional) finite-element context, once \eref{inf-sup} has been
established, one then proceeds solving
a discretized version of the  linear problem \eref{wlB}, typically in terms of a
Galerkin or Petrov-Galerkin schemes. The quantity $\|\cB w - F\|_{\Y'}$ can be viewed as
an ideal a posteriori error bound that is efficient as well as reliable. 
Such a reduction to large linear algebra problems seems less viable when using nonlinear
global trial systems. In such scenarios the reformulation of a PDE as an optimization problem becomes relevant.

There are many ways of reformulating PDEs as a variational problems. For elliptic problems, {\em energy minimization} suggests itself because
it naturally respects intrinsic model metrics. This works for a restricted class and
corresponding loss functions, however, the value of the energy at termination does not reflect the remaining estimation error. 

Alternatively for linear problems the above considerations show that $u$ solves \eref{wlB} if and only if minimizes   the mean square loss
\be
\label{minres}
u = \argmin_{w\in \X}\|\cB w - F \|^2_{\Y'},
\ee
which is not at all restricted to elliptic problems.
The advantage of this choice is that
at any stage of a minimization over a finitely parameterized hypothesis class 
$\cH(\Theta)$, determined by a budget $\Theta$ of trainable weights, by \eref{err-res}
the size of the $\|\cB\uN(\theta) -F\|_{\Y'}$ is indeed uniformly proportional to the error
$\|u-\uN(\theta)\|_\X$ and hence complies with our quest for variational correctness,
see \eqref{VC} in Section \ref{sec:intro}. Moreover, minimizing $\|\cB w - F \|^2_{\Y'}$ over
any given finitely parameterized hypothesis class retains quasi-best approximation properties:
\be
\label{nearbest}
u_\cH \in \argmin_{w\in \cH}\|\cB w-F\|_{\Y'}\quad\Rightarrow\quad \|u- u_\cH\|_\X
\le \frac{C_b}{c_b} \min_{w\in \cH}\|u- w\|_\X.
\ee

Of course, there are obvious practical issues with \eref{err-res}, and hence with \eref{minres} and \eref{nearbest}, in the finite element context as well as in our
present setting: unless $\Y=\Y'$
is self-dual (a product of $L_2$-spaces), the quantity $\|\cB u -F \|_{\Y'}$
involves a supremum over an infinite-dimensional Hilbert space
\be
\label{wrinkle}
\|\cB w - F \|_{\Y'} = \sup_{v\in\Y}\frac{(\cB w - F)( v) }{\|v\|_\Y}= \sup_{v\in\Y}\frac{b(  w,v) - F( v) }{\|v\|_\Y} \,,
\ee
which cannot be evaluated exactly. Hence, \eref{minres} in the form
\be
\label{minmax}
\theta^*\in \argmin_{\theta\in\Theta}\max_{v\in \Y}\frac{b( \uN(\theta),v) - F(v)  }{\|v\|_\Y}
\ee
is practically infeasible. This leads to our {\bf central aim} of \vspace*{0.8mm} 
\begin{center}devising {\em computable} expressions that are {\em uniformly} proportional\\[1.5mm] to
$\| \cB w - F \|_{\Y'}$, and hence to the error $\|u-w\|_\X$. 
\end{center}
\vspace*{0.7mm}
We can hope to achieve this up to unavoidable {\em data oscillation errors}, which may occur whenever the data are subjected to a projection to a finitely parameterized subset of $\Y'$.

A first natural strategy for tightly approximating $\|\cB w - F\|_{\Y'}$ is to build directly on \eref{minmax} by restricting the maximization to
some finite set $\Psi$ of trainable parameters, so that \eref{minmax} becomes 
\be
\label{GAN1}
\theta^*\in \argmin_{\theta\in \Theta}\max_{\psi\in\Psi}\frac{b( \uN(\theta),\aN(\psi)) - F(\aN(\psi)) }{\|\aN(\psi)\|_\Y}.
\ee
In the language of \emph{generative adversarial networks} (GANs) the approximation $\uN$ to the solution $u\in \X$ plays the role of a {\em generative network} while $\aN$ is the {\em discriminator} or {\em adversarial network} that is to make sure $\uN$ obeys the correct optimization criterion. 

From a practical perspective,
it seems natural to alternatingly perform gradient descent and ascent steps on the quotient in \eqref{GAN1}. This approach has been followed for special cases, namely for elliptic problems and Friedrichs systems in \cite{ZBYZ,BYZZ,Friedrichs,Pardo,Bertol,Canuto}.  One senses from the reported results that the success of this 
strategy depends strongly on the inner maximization problem being solved with sufficient accuracy. 

To understand this and to pave the way for alternatives,
recall the mapping $\cR = \cR_\Y \colon \Y' \to \Y$ defined in \eqref{Riesz}, which for $\ell \in \Y'$ is given by
\[
(\cR \ell,v)_\wtY = \ell(v),\quad v\in \wtY.
\]
Thus the \emph{Riesz lift} $\cR\ell \in \Y$ of $\ell \in \Y'$ is the solution of an {\em elliptic} variational problem
in the test space $\Y$, regardless of the type of the primal problem \eref{par}. As an immediate consequence of \eqref{Riesz},
 \[  \|\cR(\cB w - F)\|_{\Y} = \argmax_{v \in \Y} \frac{(\cB w - F)( v)}{\|v\|_\Y}. \]
Moreover, one readily checks the relation
\be 
\label{peak}
\|\cR \ell\|_\wtY^2 = \ell(\cR\ell)
= \|\ell\|_{\wtY'}^2,\qquad \ell\in \Y'.
\ee
\begin{rem}
\label{rem:Riesz}
Thus, assuming that the $\Y$-norm can be evaluated directly, finding an expression that is uniformly proportional to $\|\cB w - F\|_{\Y'}$ amounts to evaluating the Riesz lift of $\cB w - F$ up to a {\em uniform relative error}. 
\end{rem}
Specifically, whenever an approximation $v_w\in \Y$ to $\cR(\cB w-F)$ satisfies
\be
\label{rela}
\|\cR(\cB w-F)-v_w\|_\Y\le a\| \cB w-F\|_{\Y'} \quad \mbox{for some $a<1$},
\ee
one readily derives from \eref{peak} that
\be
\label{goodloss}
(1-a)\|\cB w-F\|_{\Y'}^2 \le (\cB w-F)(v_w)\le \|\cB w-F\|_{\Y'}^2,
\ee
so that $(\cB w-F)(v_w)$ is (up to quadrature) a computable variationally correct loss.
This shows though that, with an improved accuracy of the generative network, the demands on the adversarial network increases as it has to become a better and better approximation
to the Riesz lift of the residual. This explains, in particular, why
 working with fixed budgets for the generative and adversarial networks is problematic
 especially in view of an uncertain optimization success.

\section{Computable Variationally Correct Loss Functions}\label{sec:3}
\subsection{Conceptual prerequisites}\label{ssec:3.1}

Given an operator equation \eref{lB} in strong form, there   usually exist  several different pairs of trial and test spaces $\U,\V$ that give rise to a stable variational formulation.
In this section we discuss how to exploit this freedom to devise 
computable variationally correct loss functions. In what follows 
we use the notation $L_2$ to denote an $L_2$-type space by which mean in general
a product of $L_2$-spaces. Specifically, we present two scenarios 
that
can be viewed as extreme points in a spectrum of possibilities.
To explain this we suppress first any dependence on parameters $\pp$.

We emphasize that both scenarios refer to the (linear) PDE $\, \cB u=f$ 
in the form of a {\em first order system} where $\cB: {\rm dom}\,\cB\to L_2$ is a closed operator. If the PDE is initially given
as a 
higher order PDE the first step in both cases is to rewrite it as a first order
system (for simplicity again denoted by $\cB =f$), see the examples in subsequent sections.
For simplicity of exposition we assume that arising essential boundary conditions are
homogeneous. 

Recall further that the graph space of $\cB$ is defined by
$\G(\cB):= \{w\in L_2: \cB w\in L_2\}$ and $\|w\|_{\G(\cB)}^2:= \|w\|^2_{L_2}+ \|\cB w\|^2_{L_2}$.   Moreover, let $\G(\cB)_0$ denote the (closed) subspace of $\G(\cB)$ with
built-in essential boundary conditions.

We consider then two choices of the trial space $\U$,
\begin{align}
\tag{S} \U & = \G(\cB)_0 &  & \text{(strong formulation)} \qquad\qquad\label{eq:sf} \\
\tag{UW}  \U &= L_2 &  &\text{(ultra-weak formulation)} \label{eq:uwf}
\end{align}
and then look for respective test spaces $\V$ that give rise to a stable variational formulation.  

As explained in more detail later, \eqref{eq:sf} aims at having an $L_2$-least squares functional
as loss function, which is therefore easily evaluated (up to quadrature errors). 
In fact, residuals belong to $L_2$ because $\cB$ maps $\U$, by definition of the graph space, boundedly into $L_2$.

Instead, \eqref{eq:uwf} accepts an idealized loss function in terms of a  nontrivial dual norm, but,
as explained below, 
offers advantages in case of an essential parameter dependence of the trial spaces $\U_\pp$ in \eref{XY}. Moreover, in combination with discontinuous Petrov Galerkin (DPG) concepts, it still
allows one to devise computable explicitly given variationally correct loss functions.  

We pause to point out how the underlying   seemingly very different scenarios   are interrelated.
\begin{prop}
\label{rem:I} 
Let $\U\subseteq \G(\cB)_0$ be a closed subspace. Then the mapping $\cB :\U \to L_2$ is bijective if and only if
for every $f\in L_2$, there exist a $u\in \U$ such that 
\be
\label{stabvar}
b(u,v):= (\cB u)(v)= f(v)\quad \forall\, v\in L_2,
\ee
and
\be
\label{L2stabvar}
\|\cB w- f\|_{L_2}\eqsim \|w- u\|_\U,\quad w\in \U.
\ee
In this case, \eref{stabvar}
is a stable variational formulation and $\U=\G(\cB)_0$.  %
\end{prop}
\begin{proof}
 First, let $\cB :\U \to L_2$ be bijective. Then by definition of the graph norm, $\|\cB w\|_{L_2}\le \|w\|_{\G(\cB)}$, $w\in \U$.  Since $\U\subseteq \G(\cB)_0$ one has $\cB(\U)\subset L_2$. Since $\cB$ maps $\U$ onto $L_2$, boundedness of $\cB^{-1}$ follows from the 
open mapping theorem. Moreover $\U= \G(\cB)_0$, since otherwise a non-trivial orthogonal complement would contradict injectivity. The Babu\v{s}ka-Ne\v{c}as Theorem ensures then that the bilinear form
$(\cB w)(v): \U\times L_2\to \R$ yields stable variational formulation, which implies \eqref{stabvar} and
\eref{L2stabvar}.
The converse is a trivial consequence of the Babu\v{s}ka-Ne\v{c}as Theorem.
\end{proof}

 \begin{rem}
 \label{rem:PINN}
 When using the strong formulation \eref{eq:sf} under the hypotheses in Proposition \ref{rem:I}, a Monte-Carlo approximation of the
 residual in an $L_2$-space, as in   standard versions of PINN,  does not incur any variational crime.
 \end{rem}

 We proceed with a few comments on scenario \eref{eq:uwf} concerning ultra-weak formulations.
 \begin{prop}
 \label{rem:uw}
 Let $\cB^*$ denote the formal adjoint of $\cB$, i.e., for smooth functions
 $\phi$ with compact support in the computational domain $\Omega$ one has
 $\lll \cB w, \phi\rr_{L_2}= \lll w,\cB^*\phi\rr_{L_2}$. Let
 \be 
 \label{V}
 \V := {\rm clos}_{\|\cB^*\cdot\|_{L_2}}\big\{\phi\in \G(\cB^*):
 \lll \cB w, \phi\rr_{L_2}- \lll w,\cB^*\phi\rr_{L_2} =0\,\,\forall\, 
 w\in \U=\G(\cB)_0\big\}.
 \ee
 The following statements are equivalent:
 \begin{enumerate}[{\rm(i)}]
 \item
 $\cB^*: \V\to L_2$ is bijective.
 \item
 $\|v\|_\V:= \|\cB^* v\|_{L_2}$ is a norm on $\V$ and the \emph{ultra-weak} variational formulation
 \[
 b(u,v):=(\cB u)(v)=f(v), \quad \forall\,v\in \V,
 \]
 is stable with constants $c_b=C_b=1$.
 \end{enumerate}
 \end{prop}
\begin{proof}
 If (i) holds, $\|\cdot\|_\V$ is a norm. Obviously the bilinear form $b(\cdot,\cdot):L_2\times \V\to \R$ satisfies $|b(u,v)|= |(\cB u)(v)|= |(\cB^* v)(u)|\le \|\cB^*v\|_{L_2}\|u\|_{L_2}$, which means that $C_b\le 1$. Moreover, 
 $$
 \inf_{v\in \V}\sup_{w\in L_2}\frac{b(w,v)}{\|v\|_\V\|w\|_{L_2}}= \inf_{v\in \V}\frac{
 \|\cB^*v\|_{L_2}}{\|\cB^* v\|_{L_2}}= 1
 $$
 and for each $w\in L_2$ there exists a $v_w\in \V$ such that $b(w,v_w)=
 (\cB^* v_w)(w)\neq 0$ because $\cB^*$ is surjective. This confirms the claim. The converse follows directly from
 the Babu\v{s}ka-Ne\v{c}as Theorem.
 \end{proof}
 
 The key in both scenarios is to establish bijectivity of a linear operator 
 as a mapping to $L_2$, keeping in mind that $\cB$ and $\cB^*$ are of the same ``type'',
 so the difficulties in both cases are very similar. Moreover, as a consequence
 of Proposition \ref{rem:uw}(ii),
 $\cB$ is   an isomorphism from $L_2$ to $\V'$. Hence, its dual $\cB'$
 (by self-duality of $L_2$) is an isomorphism from $\V'$ to $L_2$ which agrees
 with $\cB^*$ on dense smooth subsets. So, roughly, the dual of $\cB^*$ agrees with
 $\cB$ as a mapping from $L_2$ to $\V'$ (which could be viewed as a continuous
 extension of the operator $\cB:\U\to L_2$, defined in scenario \eref{eq:sf}).
 
 Returning to {\em parameter dependence}, whenever Proposition \ref{rem:uw}(ii) applies,
 when employing the optimal test norm, the stability is {\em uniform} in $\pp$ 
 warranting an optimal condition of the parametric problems as well as of the lifted one
 \eref{wlB}, see Theorem \ref{thm:suff}. 
 
 In scenario \eref{eq:sf}, {\em uniform} boundedness of $\cB=\cB_\pp$ with respect $\pp\in \mD$
 is ensured by definition. If in addition
   the graph spaces $\G(\cB_\pp)$ agree with equivalent norms {\em for all} $\pp$
 we conclude that \eref{inf-sup-pp} is valid which, by Theorem \ref{thm:suff}
 implies well-posedness of the lifted problem \eref{wlB}. One 
   then still ends up with  a computable  $L_2$-least squares functional as variationally correct loss function. 
   
 If on the other hand, the graph spaces $\G(\cB_\pp)$ depend {\em essentially} on $\pp\in\mD$, by which we mean that even as sets  the $\G(\cB_\pp)$ vary with $\pp$,
and thus the membership of approximate solutions of the fiber problems to the respective 
parameter dependent trial spaces
is not robust. We see this as a reason to opt for scenario \eref{eq:uwf}, based on 
{\em ultra-weak formulations}. How to proceed in this case can be summarized as follows.
\begin{rem}
\label{rem:roadmap}
	For a wide variety classes of PDEs, one can proceed as follows to obtain suitable test spaces $\V_\pp$ for the choice $\U_\pp$: 
If \eref{lB} is not already given in this form, rewrite it as a \emph{first order system} of PDEs (introducing auxiliary variables if needed), again denoted for simplicity by $\cB u=f$. We assume homogeneous essential boundary conditions, so that $\cB$ is indeed linear. 
With the norm
\be
\label{range}
\|v\|_{\V_\pp} :=  \|\cB_\pp^* v\|_{L_2 },
\ee
define the spaces $\V_\pp$ for each parameter $\pp$ as closures with respect to this norm according to \eref{V}.
As a crucial step, one then needs to show that for each $\pp\in\mD$, the adjoint $\cB_\pp^*$ is a bijective mapping from $\V_\pp$ onto $L_2$. 

Once this has been confirmed, one defines $\X:=L_2(\mD\times \Omega)$, $\Y:= L_2(\mD;(\V_\pp)_{\pp\in\mD})$ and $b \colon \X\times \Y \to \R$ as in \eqref{eq:paramweak}, where by integration by parts, all derivatives are applied to the test function $v\in \V_\pp$. For the norm of $\Y$, we thus obtain
\be
\label{range}
\|v\|_\Y =  \|\cB^* v\|_{L_2(\mD\times\Omega) }.
\ee
Given $F\in \Y'$, finding $u\in \X:= L_2(\mD\times \Omega) $ satisfying the {\em ultra-weak} formulation
\be
\label{gen}
b(u,v)= F(v),\quad v\in \Y,
\ee
is then a well-posed problem in the sense of \eref{inf-sup}.
\end{rem}
Once Proposition \ref{rem:uw}(i) has been verified,  Proposition \ref{rem:uw}(ii) applies and says that the
fiber problems have condition number {\em equal to one}, uniformly in $\pp\in\mD$, that is,
\be
\label{equal}
\|u - w\|_\X = \|\cF(w)\|_{\Y'}.
\ee
We therefore call the right hand side an ``ideal residual loss'' because the respective
norms of errors and residuals agree and the variational problem has been
{\em optimally preconditioned} already on the continuous level. 
Note also that in such ultra-weak formulations of the form \eref{gen} essential boundary conditions
can no longer be imposed on elements in $L_2$. In fact, essential boundary conditions
become
natural ones.

On the other hand, one still faces   the problem that
the resulting ideal residual loss function $\|\cB u-F\|_{\Y'}^2$ involves a nontrivial dual norm. 

What we have gained though is that the ultra-weak formulation \eref{eq:uwf} is the perhaps most convenient
starting point for deriving stable {\em discontinuous Petrov-Galerkin} (DPG) formulations.
For the convenience of the reader we briefly recall in   Section \ref{ssec:dpg}  the  relevant DPG concepts. 
In combination with the hybrid representation format as in Section \ref{ssec:hybrid}, which is discussed further in Section \ref{ssec:DNNhybrids}, DPG formulations will be shown below to  lead then to computable variationally correct  loss functions.
 
Stable DPG formulations have by now been established for a wide range PDE models including Maxwell's equations, dispersive models, or singularly perturbed problems, primarily in the context
of DPG methods based on well-posed variational formulations of the above type; see, for example, \cite{BDS,CDG,CDW,DGM,DHSW,DGNS,GopSep}.
 We proceed discussing examples for both strong and ultra-weak formulations.

\subsection{Strong formulation: First order system least-squares methods in $L_2$}
We consider two different model problems where $\U$ can be chosen as a parameter-independent graph space as in \eqref{eq:sf}.

\subsubsection{An elliptic model problem} \label{ssec:3.2}%
\newcommand{\divv}{{\rm div}}
\newcommand{\wz}{\mathfrak{z}}\newcommand{\bu}{\mathbf{u}}
Although one could include in what follows lower order terms we focus for simplicity
of exposition on the classical Poisson equation (without reaction or convection term)
with homogeneous Dirichlet conditions
\be
\label{Poisson}
- {\rm div}(a(\pp)\nabla u)=f \quad \mbox{on $\Omega$ and $u|_{\partial\Omega}=0$},
\ee
where we consider parameter-dependent diffusion coefficients satisfying the {\em uniform
ellipticity conditions} (UE):
\be
\label{UE}
\exists\,r,R,\quad  0<r\le R<\infty \quad \mbox{such that }\,\, r\le a(x,\pp)\le R,\quad (x,\pp)\in \Omega\times \mD.
\ee
We   suppress in what follows first the dependence on the parameter $\pp$ where it does not matter.
We say that
$u\in H^1_0(\Omega)$ is a solution of \eref{Poisson} if 
\be
\label{standard}
\langle a\nabla u,\nabla v\rangle_{L_2(\Omega)} = f(v),\quad v\in H^1_0(\Omega),
\ee
which is known to be stably solvable for any $f\in H^{-1}(\Omega):= (H^1_0(\Omega))'$.
 Next recall that 
$$
H({\rm div},\Omega):= \{\bw\in L_2(\Omega;\R^d): {\rm div}\,\bw\in L_2(\Omega) \},\quad 
\| \bw\|^2_{H({\rm div},\Omega)}:= \|\bw\|^2_{L_2(\Omega;\R^d)}+ \|{\rm div}\,\bw\|^2_{L_2(\Omega)}.
$$
Also
recall that every $f\in H^{-1}(\Omega)$ can be written as  $f = f_2+ \operatorname{div}\bbf_1$, for some $f_2\in L_2(\Omega)$, $\bbf_1\in L_2(\Omega;\R^d)$.
Hence, if $ u\in H^1_0(\Omega)$ is
the solution of \eref{standard} for this $f$, one has
\be
\label{f1f2}
 \boldsymbol{\sigma} := a\nabla  u+\bbf_1\in L_2(\Omega) \quad \mbox{solves}\quad -\operatorname{div}\boldsymbol{\sigma} = f_2.
\ee
In other words, defining 
\be
\label{mixed} 
\cB_\pp: \ug = [\boldsymbol{\sigma},u] \mapsto 
\cB_\pp \ug = \left(\begin{matrix}
 \id & -a(\pp)\nabla\\
-\divv & 0
\end{matrix} 
\right)\left(\begin{matrix}
	\boldsymbol{\sigma}\\
u
\end{matrix}
\right)= \left(\begin{matrix}   \boldsymbol{\sigma} - a(\pp)\nabla u\\
-\operatorname{div}\boldsymbol{\sigma}
\end{matrix}\right),%
\ee
$\ug:=[\nabla  u, u]$ solves $\cB_\pp \ug = \mathfrak{f} = (\bbf_1,f_2)$. Note that $\nabla u\in
H(\divv;\Omega)$ 
 and $ u\in H^1_0(\Omega)$. Thus 
\be
\label{UPoisson}
\U := H(\divv;\Omega)\times H^1_0(\Omega)
\ee
is a Hilbert space contained in $\G(\cB)_0$, and the previous observation
shows   that $\cB_\pp:\U\to L_2$ 
is surjective.
As shown in \cite{FOSL}, the symmetric bilinear form
\be
\label{BigB}
B(\ug,\wg;\pp):= \langle\cB_\pp\ug,\cB_\pp \wg\rangle_{L_2(\Omega)} : \U\times \U\to \R
\ee
is $\U$-elliptic, that is, 
\be
\label{Uequiv}
\|\wg\|^2_{\G(\cB)_0}=B(\wg,\wg;\pp)\eqsim \|\wg\|^2_\U,\quad \wg\in \U, \quad \mbox{uniformly in $\pp$ obeying \eref{UE}}.
\ee
This implies, in particular,   injectivity and hence bijectivity of of $\cB_\pp:\U\to L_2.
$ 

Hence Proposition \ref{rem:I} applies and says that $\U= \G(\cB)_0$ and $\cB=\cB_\pp:
\G(\cB)_0  \to L_2$,
defined by \eref{mixed},  is bounded and boundedly invertible. This, in turn says that
\be
\label{opt}
\ug(\pp) =\argmin_{\wg\in \U}\|\cB_\pp \wg- \mathfrak{f}\|^2_{L_2}
\ee
is the unique solution of $\cB_\pp \ug(\pp)= \mathfrak{f}$ and, on account of \eref{Uequiv},
\[
	\|\cB_\pp \wg - \mathfrak{f}\|_{L_2}\eqsim \|\ug-\wg\|_\U.	
\]
holds uniformly in $\pp$.  

The formulation (\ref{opt}) can be regarded as a particular instance of the framework of Section 
\ref{sec:residual}, with $\U_\pp=\U$ as in \eqref{UPoisson}, $\V_\pp=\V = L_2$, and the $\pp$-dependent bilinear form
\be\label{eq:elliptparamform}
b(\ug,\mathfrak{v};\pp):= (\cB_\pp \ug)(\mathfrak{v}) = \langle \boldsymbol{\sigma} - a(\pp)\nabla u,\boldsymbol{\tau}\rangle_{L_2(\Omega)}
-\langle\operatorname{div}\boldsymbol{\sigma},v\rangle_{L_2(\Omega)}, \quad \ug \in \U, \, \mathfrak{v} = [\boldsymbol{\tau},v] \in L_2.
\ee
In fact, $\ug(\pp)$ is characterized by the normal equations
$B(\ug,\wg;\pp)= \langle\mathfrak{f},\cB_\pp\wg\rangle_{L_2(\Omega)}$, $\wg\in \U$. Again, since $\cB_p:\U\to L_2$
is an isomorphism, this is equivalent to 
$$
b(\ug,\mathfrak{v};\pp)= B(\ug,\cB_\pp^{-1}\mathfrak{v};\pp)= 
\langle\mathfrak{f},\mathfrak{v}\rangle_{L_2(\Omega)} = 
\mathfrak{f}(\mathfrak{v}),
$$
posed over $\U\times L_2$.
\begin{rem}
\label{rem:noness}
By \eref{Uequiv} and \eref{UE}, the above well-posedness is uniform in $\pp$
so that the above trial and test spaces $\U=H(\divv;\Omega)\times H^1_0(\Omega)$ and $\V=L_2$ are indeed independent of $\pp$.
 Thus, Theorem \ref{thm:suff} applies (see Remark \ref{rem:justified}) and we can
readily obtain from \eref{eq:elliptparamform}  a weak formulation with bilinear form $b(\cdot,\cdot): \X\times \Y \to \R$ with 
$\X= L_2(\mD;\U)$, $\Y'= L_2(\mD;L_2) = L_2(\mD)\times L_2$, where $L_2= L_2(\Omega;\R^d)\times L_2(\Omega)$.
\end{rem}

\begin{rem}
\label{rem:equiv}
We emphasize that the formulations \eqref{standard} and \eqref{opt} are equivalent in the following sense:
As noted above, if $u$ solves \eref{standard} then $\ug =[\nabla u,u]$ solves $\cB_\pp\ug = \mathfrak{f}$.
Conversely, when $\ug=[\boldsymbol{\sigma},u]$ solves $\cB_\pp\ug=\mathfrak{f}$  then $u$ solves 
\eref{standard} with $\boldsymbol{\sigma} =  \nabla u$.
\end{rem}

\subsubsection{A parabolic model problem} \label{ssec:3.3}
We close the discussion of scenario \eqref{eq:sf} by applying 
  the above principles apply   to a time dependent
problem, leading to space-time
variational formulations.
We recall the following standard well-posed space-time variational formulation for 
parabolic initial boundary value problems: for $I=(0,T)$ find 
$$
u\in L_2(I;H^1_0(\Omega))\cap H^1(I;H^{-1}(\Omega))
$$
such that
\be
\label{parab}
\int_{I\times\Omega} \partial_t u v + a\nabla_x u\cdot\nabla_xv\,dt\,dx= \int_{I\times\Omega}fv\,dt\,dx,\quad \forall\, v\in  L_2(I;H^1_0(\Omega)),
\ee
where $f\in L_2(I;H^{-1}(\Omega))$ and $a$ may depend also on parameters $\pp$
such that \eref{UE} holds. It is well-known that this (asymmetric) variational formulation is stable, which remains true when the roles of trial and test space are
interchanged, with a terminal condition at $T$ on the test space, see e.g. \cite{Stev}.

Here both trial and test space involve non-trivial dual norms which 
motivates interest in computationally more friendly formulations. 
A formulation similar to \eqref{mixed} for the elliptic case has been introduced in \cite{FK}.
Writing $f\in L_2(I;H^{-1}(\Omega))$ again as $f=f_2+ \divv_x\,\bbf_1$, 
with $\bbf_1\in L_2(I\times \Omega;\R^d)$, consider the first-order system
\be
\label{Bparab}
\cB\ug := \left(\begin{matrix}
-\boldsymbol{\sigma}- a\nabla_x u\\
\partial_t u+\divv_x\boldsymbol{\sigma}\\
u(0,\cdot)
\end{matrix}\right)= 
\left(\begin{matrix}
\bbf_1\\
f_2\\
u_0
\end{matrix}\right),\quad u_2|_{I\times \partial\Omega}=0.
\ee
In this case, as shown in \cite{GS}, the graph space is isomorphic to
\be
\label{pargraph}
\U =  \big(L_2(I\times\Omega;\R^d)\times L_2(I;H^1_0(\Omega))  \big)\cap H(\divv;I\times\Omega),
\ee
where $H(\divv;I\times\Omega)$ is defined in terms of the space-time divergence $\operatorname{div} [\boldsymbol{\sigma},u] = \operatorname{div}_x \boldsymbol{\sigma} + \partial_t u$.
In particular, no explicit treatment of dual spaces is required in this formulation.
As shown in \cite{GS}, the mapping $\cB$ is boundedly invertible from 
$\U$ to the $L_2$-space
\[
\V = L_2(I\times\Omega;\R^d) \times L_2(I\times \Omega)\times L_2(\Omega)
\]
so that one is in the same situation as in the elliptic case, that is, for any 
$F= (\bbf_1,f_2,u_0)\in \V$ there exists a unique $\ug\in \U$ such that
$\cB\ug = F$ and
$$
\ug= \argmin_{\wg\in\U}\|\cB\wg-F\|^2_{\V},\quad \|\ug-\wg\|_\U
\eqsim \|F- \cB\wg\|_{\V},\,\, \wg\in \U.
$$
Analogous remarks on the elliptic case   in the previous section apply to the current situation as well.

\subsection{ Neural DPG methods for ultra-weak formulations \eref{eq:uwf}}\label{ssec:DPG}

\subsubsection{Linear transport equations}\label{sssec:transport}
 Let $\Omega\subset\R^{d_x}$
be a domain and consider 
\be
\label{transport}
\cB_\pp u(x,\pp):=\bb(x,\pp)\nabla u(x) + c(x,\pp)u(x) = f(x,\pp)  \quad \mbox{in}\,\,\Omega,
\quad u|_{\Gamma_{\rm in}(\pp)}=g,
\ee
where $\Gamma_{\rm in}(\pp)=\operatorname{clos} \{x\in \partial\Omega: \bn(x)\cdot\bb(x,\pp)< 0\}$ denotes
the \emph{inflow-boundary} of $\Omega$. Here we assume that $\partial\Omega$ is smooth
enough so that the outer normal $\bn(x)$ exists a.e. on $\partial\Omega$. Analogously, we define the  {\em outflow boundary} 
\[  \Gamma_{\rm out}(\pp)= \operatorname{clos}{ \{x\in \partial\Omega: \bn(x)\cdot\bb(x,\pp)> 0\} }, \]
and furthermore the {\em characteristic boundary} $\Gamma_{0}(\pp)= \partial\Omega\setminus ( \Gamma_{\rm in}(\pp) \cup \Gamma_{\rm out}(\pp) )$.
 Note, that the formulation \eref{transport}
includes the case of instationary problems, where one simply views the first component of $x$ as 
the time variable and takes the first component $\bb_1(x,\pp)$ to be equal to one. 
We postpone specifying conditions on the convection field and first put this example into context.

Although the concepts used in what follows apply in much larger generality,
we discuss this specific example for several reasons. First it is of interest, for instance, as a core constituent
of more involved kinetic models. Second, despite  its ``analytical simplicity'', it exhibits
several difficulties when dealing with parameter dependent convection fields.
In fact, when $\Omega$ is a polyhedral domain, the inflow-boundary may flip
abruptly when the convection direction passes through specific parameter instances. This entails an intrinsic discontinuity
with respect to the parameter dependence. 

To that end, since $\cB_\pp$ has already order one, one could adopt scenario \eref{eq:sf},   and look for solutions in the graph space of $\cB_\pp$ comprised of those functions $w$ in $L_2(\Omega)$ for wich $\cB_\pp w$ is also in $L_2(\Omega)$. 
For a function $w$ to belong to the graph space requires directional derivatives 
of $w$ along characteristics to belong to $L_2(\Omega)$ which allows for discontinuities along the characteristics. Thus, even under small perturbations of the convection field the 
graph space changes not only with respect to the norm but already as a set. 
Therefore, we opt for ultra-weak formulations according to scenario \eref{eq:uwf}    
and seek solutions just in $L_2(\Omega)$ which
is indifferent under varying convection.

As shown below, the fiber test spaces then depend   on $\pp$ in an essential way, 
which is more tolerant to
perturbations, see \cite{BDS}.
Thus, ``lifting'' the family of fiber problems \eref{transport}
to a single variational problem over a pair of classical Bochner spaces would not work
and one needs to resort to the type of direct integrals defined in
 \eref{XY}, in particular, to arrive at well-defined dual norms underlying an ideal
 loss function.

To exhibit appropriate weak formulations of the fiber problems \eref{transport}
we briefly sketch the relevant arguments from \cite{BDS,DHSW}. 

We consider first a fixed parameter $\pp$
and recall for convenience some relevant facts from \cite{DHSW,BDS}. We follow Proposition \ref{rem:uw} to determine $\cB_\pp^*$ and the test space $\V_\pp$. 
To that end,  integration by parts yields
\begin{align}
\label{Green}
0 &= \int_\Omega -u(x)\big[\Div(\bb(x,\pp) v)(x) - c(x,\pp)v(x)\big]\,dx
\nonumber\\
&\qquad\qquad\qquad\qquad
- \int_{\partial\Omega}\bn(x)\cdot \bb(x,\pp) u(x)v(x)\,d\Gamma(x) -\int_{\Omega}f(x)v(x)\,dx.
\end{align}
Notice that when $u\in L_2(\Omega)$ is all we know, traces are not defined. But when inflow-boundary conditions
$g$ are imposed, \eref{Green} can be rewritten as
\begin{align}
\label{uwtrans}
0 &= - \int_\Omega u(x) \big[\Div(\bb(x,\pp) v)(x) - c(x,\pp)v(x)\big]\,dx\nonumber
\\ 
&\qquad\qquad\qquad\qquad+ \int_{\Gamma_{\rm in}(\pp)}|\bn\cdot\bb(\cdot,\pp)|g(\cdot,\pp)v(\cdot)\,d\Gamma(x)
- \int_{\Omega}f(x)v(x)\,dx\nonumber\\
&=: \ell(v;\pp)-  b(u,v;\pp), 
\end{align}
where 
\be
\label{ell}
\ell(v,\pp):= \int_{\Gamma_{\rm in}(\pp)}|\bn\cdot\bb(\cdot,\pp)|g(\cdot,\pp)v(\cdot)\,d\Gamma(x)
- f(v), %
\ee
provided that the test function $v$ vanishes on the outflow boundary $\Gamma_{\rm out}(\pp)$. Defining further
\be
\label{Vpp}
\V = \V_\pp := {\rm clos}_{\|\cdot\|_{\V_\pp}}\Big\{v\in C^1(\overline\Omega): v|_{\Gamma_{\rm out}(\pp)}=0\Big\},
\ee
where
\be
\label{Vnorm}
 \|v\|_{\V_\pp}:= \norm{\Div(\bb(\cdot,\pp) v) - c(\cdot,\pp)v}_{L_2(\Omega)},
\ee
follows exactly the receipe \eref{V} in Proposition \ref{rem:uw} with the norm 
given in of Proposition \ref{rem:uw}(ii).

For the linear functional $\ell$ to belong to $\V_\pp$ we need that $g, v|_{\Gamma_{\rm in}(\pp)}$ belong  the weighted $L_2$ space
$L_2\big(|\bb(\cdot,\pp)\cdot\bn(\cdot)|,\Gamma_{\rm in}(\pp)\big)$.
It is well-known that the elements of $\V_\pp$ indeed have a trace in this space
(see the discussion in \cite{DHSW}).
Note, that essential inflow boundary conditions in the strong formulation have become natural ones appearing as part of the right hand side.

Moreover, it has also been shown in \cite{BDS} that $\cB^*_\pp = \Div(\bb(\cdot,\pp) v) - c(\cdot,\pp)v$
is indeed an isomorphism as a mapping from $\V_\pp$ onto $L_2(\Omega)$ under mild assumptions on the convection field $\bb(\cdot,\pp)$ which we assume to be valid in what follows. 

It is difficult to characterize well-posedness of \eref{uwtrans} in terms of concise
properties of the convection field and one usually has to be content with sufficient conditions, see e.g. related discussions in \cite{BDS,DHSW}. 
Throughout the remainder of this section we work under the following assumption
that guarantee, in particular, well-posedness of \eref{uwtrans} for each $\pp\in\mD$.

\begin{assumption}\label{ass:vectorfield}
	We assume for almost all $\pp\in\mD$ that $\mathbf b(\cdot;\pp)$ is Lipschitz on $\bar\Omega$ and $|\mathbf b(\cdot;\pp)|\ge \beta>0$ for some fixed $\beta$, that is,
	the convection is not allowed to degenerate. Moreover, we either require that for
	each $\pp\in \mD$
	that $\bb(\cdot;\pp)$ is in $C^1(\Omega;\R^d)$ or there exists a positive constant $\kappa$ such that $c -\frac 12 {\rm div}\,\bb(\cdot;\pp)\ge \kappa$.
	Finally, we assume that $\mathbf b$ is jointly measurable in $x$ and $\pp$.
\end{assumption}
These conditions ensure, in particular, that for $v\in \V_\pp$ one has $\|v\|\le C \|\cB^*_\pp v\|$ which means that $\cB^*_\pp$ is a closed operator. Since the structure of
$\cB_\pp$ is the same, $\cB_\pp$ is closed for each $\pp\in\mD$ as well.
The main consequences can be summarized as follows.
\begin{rem}
\label{rem:pstab}
Employing the test-norms \eref{Vnorm}, the family of fiber problems are well-posed
for the pair  $\U_\pp =L_2(\Omega)$, $\V_p$, given by \eref{Vpp}, for each $\pp\in\mD$, i.e.,  
\eref{inf-sup-pp} is valid with $c_b=C_b=1$, and thus uniformly in $\pp$.
\end{rem}

\begin{rem}
\label{rem:bsat}
Under Assumption \ref{ass:vectorfield} on $\bb(\cdot,\pp)$, Property \ref{ass:bil} is satisfied. %
\end{rem}

 To be able to
formulate  the corresponding lifted problem and to invoke Theorem \ref{thm:suff} 
(and take advantage of the tight error-residual relation \eref{equal}), the following
claim is crucial.
\begin{theorem}
\label{thm:dint}
Under Assumption \ref{ass:vectorfield}, the direct integral $\Y=L_2(\mD;(\V_\pp)_{\pp\in\mD})$ is a well-defined Hilbert space with dual $\Y' = L_2(\mD;(\V'_\pp)_{\pp\in\mD})$.
\end{theorem}
With this result at hand, the concepts presented in Section \ref{sec:residual}
can be applied.

In view of Theorem \ref{thm:mes}, the claim in Theorem \ref{thm:dint} is an immediate consequence
of  the following Lemma.

\begin{lemma}
\label{lem:V_pDirectIntegral}
There is a fundamental $\mu$-measurable sequence $\prod_{\pp\in\mD}\V_\pp$ such that the corresponding direct integral $\Y=L_2(\mD,(\V_\pp)_{\pp\in\mD})$ satisfies 
 	\[  L_2(\mD,(\V_\pp)_{\pp\in\mD})=\mathrm{ clos}_{\|\cdot\|_{\Y}}\{ v\in L^2 (\mD\times \Omega )\colon v(\pp,\cdot)\in C^1(\Omega),  v(\pp,\cdot)|_{\Gamma_{\rm out}(\pp)}=0 \text{ for a.e.\ $\pp\in\mD$} \}  \]
 	with 
 	 \[  \|v\|_\Y^2 = \int_\mD \|  v(\pp,\cdot)\|^2_{\V_\pp}\,d\mu(\pp). \]

\end{lemma}

We defer the proof of this lemma to Appendix \ref{appendix:A}.
 Here we briefly comment on the
main steps. A first auxiliary ingredient is the notion of graph space which we briefly
recall, abbreviating at times
  $\|\cdot\|:= \|\cdot\|_{L_2(\Omega)}$, when the reference to the domain is clear
$$
\G^*_\pp:= \G(\cB^*_\pp) = \{ v\in L_2(\Omega): \|v\|^2+ \|\cB^*_\pp v\|^2:= \|v\|^2_{\G(\cB^*_\pp)}<\infty\}. 
$$
By our previous comments, $\cB^*_\pp$ is closed and hence $\G^*_\pp$ is a Banach space.

An important observation is that the Sobolev space $H^1(\Omega)$ is dense in 
every $\G^*_\pp$ for each $\pp\in \mD$. A fundamental system for $\prod_{\pp\in\mD}\V_\pp$
can then be constructed by first taking (any) fundamental system in $H^1(\Omega)$
and multiplying each element in this system by a judiciously chosen $\pp$-dependent {\em clipping function} whose properties allow one to show that the resulting system is fundamental in $\prod_{\pp\in\mD}\V_\pp$.

\subsubsection{Discontinuous Petrov-Galerkin discretization}\label{ssec:dpg}

For the convenience of the reader we briefly recall features of discontinuous Petrov-Galerkin (DPG) methodology that matter for the present purposes.
 (DPG) methods are relevant when well-posedness necessitates 
$\U\neq \V\neq L_2$. It is derived from a {\em well-posed} variational formulation 
\be
\label{original}
\mbox{find $u\in\U$ such that:}\qquad b(u,v)=f(v),\quad v\in \V.
\ee
For the moment, we
 suppress   for convenience dependence on any parameters $\md\in \mD$.
Although not necessary, DPG methods are most conveniently explained when \eref{original} is an ultra-weak formulation which is henceforth assumed, see scenario \eref{eq:uwf}.

 Given   a (shape regular) partition or mesh $\cT$  of
the spatial domain $\Omega$, the first step is to derive from \eref{original} 
an additional {\em mesh-dependent} well-posed variational formulation -- first on the continuous level 
\be
\label{meshdep}
\mbox{find $\ug\in\U_\cT$ such that:}\qquad b_\cT(\ug,v)=f(v),\quad v\in \V_\cT,
\ee
where its ingredients are explained in a moment. First, elementwise integration 
by parts introduces trace terms which for $u\in L_2(\Omega)$ are not defined. 
This necessitates introducing, as  auxiliary unknowns, {\em skeleton components} living on $\partial\cT:=\bigcup_{K\in\cT}\partial K$. As the notation $\ug$ indicates, the unknown $\ug$ has therefore
several components $\ug=(u_b,u_s)\in \U_\cT= \U_{\cT,b}\times \U_{\cT,s}$ where 
$$
\U_{\cT,b}:=  \prod_{K\in\cT} L_2(K),\quad \U_{\cT,s}=\Big\{
u_s\in L_2(\partial\cT): \|u_s\|_{\U_{\cT,s}}:= \inf_{w\in \G(\cB): w|_{\partial\cT}= u_s}\|\cB w\|_{L_2(\Omega)}<\infty\Big\}.
$$
The bilinear form $b_\cT(\ug,v)$ has then the structure
\be
\label{bT}
b_\cT(\ug,v)= \sum_{K\in\cT} b_K(u_b,v) + \lll u_s,v\rr_{\partial K} =: \sum_{K\in\cT} B_K(\ug,v),
\ee
and finally, perhaps most importantly, the test space $\V_\cT$ is a {\em broken space}
\be
 \label{VK}
\V_\cT = \prod_{K\in\cT} \V_K,\quad \|v\|^2_{\V_\cT} = \sum_{K\in\cT}
\|\cdot\|^2_{\V_K},
\ee
where the $\V_K$, $K\in \cP$, are ``localized'' versions of the test-space
$\V$ in the underlying ``conforming'' formulation \eref{original}. 
A simple but crucial consequence of this latter feature is the product structure of the
Riesz lift $\cR_\cT:\V_\cT'\to\V_\cT$
\be
\label{RT}
\V_\cT'=\prod_{K\in\cT}\V_K',\qquad \cR_\cT = \prod_{K\in \cT}\cR_K,
\ee
where as before  $\lll \cR_\cT \ell,v\rr_{\V_\cT}=
\ell(v)$, $v\in \V_\cT$, and $\lll \cR_K \ell,v\rr_{\V_K}=
\ell(v)$, $v\in \V_K$. 

Once well-posedness of \eref{meshdep} has been established (to be assumed in what follows), upon
introducing the DPG residual $\cF_\cT(\wg)(v):= f(v)- b_\cT(\wg,v)$, we have by
 \eref{VK}  
$$
\|\cF_\cT(\wg)(v)\|_{\V_\cT'}^2 = \sum_{K\in \cT}\|\cF_K(\wg)\|^2_{\V_K'},\quad 
\|\cF_K(\wg)\|_{\V_K'}= \sup_{v\in \V_K}\frac{f(v)- b_K(\wg,v)}{\|v\|_{\V_K}}.
$$
Thus, 
well-posedness of \eref{meshdep} says that the exact solution $\ug\in \U_\cT$ of \eref{meshdep}
satisfies
\be
\label{resP}
\|\ug- \wg\|^2_{\U_\cT} \eqsim \|\cF_\cT(\wg)(v)\|^2_{\V_\cT'}
= \sum_{K\in \cT}\|\cF_K(\wg)\|^2_{\V_K'} ,\quad \wg\in \U_\cT.
\ee
Although this is the starting point for a posteriori error estimation it is in the above
form not practically feasible since each term $\|\cF_K(\wg)\|^2_{\V_K'}$ still involves a supremization over the infinite-dimensional local test space $\V_K$. 
The key step towards a practical DPG method is to identify for a given
  trial spaces $\U_{\cT}^\delta\subset \U_\cT$ and a data space $\mathbb{F}^\delta$ of piecewise polynomials on $\cT$   of a {\em fixed maximal order}, local {\em test-search spaces} $\wt\V^\delta_K$, $K\in\cT$, such that 
\be
\label{discr-inf-sup}
\inf_{\wg^\delta\in \U^\delta_{\cT},f^\delta\in \mathbb{F}^\delta}\sup_{v^\delta\in \wt\V^\delta_\cT}\frac{b_\cT(\wg^\delta,v^\delta)- f^\delta(v^\delta)}{\|\wg-\cB_\cT^{-1}f^\delta\|_{\U_\cT}\|v^\delta\|_{\V_\cT}}\ge \beta>0,
\ee
holds uniformly in $\cT$. Of course, by our assumption on well-posedness of \eref{meshdep},
\eref{discr-inf-sup} holds trivially when $\wt\V^\delta_\cT$ is replaced by $\V_\cT$.
Note that, upon taking $f^\delta=0$ \eref{discr-inf-sup} requires uniform inf-sup stability of the Petrov Galerkin scheme with trial space $\U_\cT^\delta$ and test space
$$
\V^\delta_\cT := \prod_{K\in\cT} \mT^\delta_K \U^\delta_\cT, \quad \lll \mT^\delta_{B_K}\wg^\delta,v^\delta\rr_{\V_K}= B_K(\wg^\delta,v^\delta),\,\, \forall \, v^\delta \in \wt\V^\delta_K.
$$
Thus, calculating $\mT^\delta_{B_K}\wg^\delta$ amounts to solving a symmetric positive definite system of size ${\rm dim}\,\wt\V^\delta_K$. In complete analogy the {\rm trial-to-residual} map $\mT^\delta_{\cF_K}$ is defined by 
$$
\lll \mT^\delta_{\cF_K}\wg^\delta,v^\delta\rr_{\V_K}= \cF_K(\wg^\delta,v^\delta),\quad v^\delta\in \wt\V^\delta_K,\quad K\in\cT.
$$

Hence the game is to choose, on the one hand, the $\wt\V^\delta_K$ large enough to
satisfy \eref{discr-inf-sup}. On the other hand, one wishes to keep 
$\wt n:= {\rm dim}\,\wt\V^\delta_K$ as small as possible, ideally {\em uniformly bounded}
with respect to $K,\cT$. If this is the case, although quite expensive in quantitative terms, the size of discrete DPG systems still remains uniformly proportional to $\#\cT$
which is a cost-lower bound for all discretizations based on the partition $\cT$.
This program has by now been 
carried out for a wide scope of PDE models, see, e.g., \cite{CDG,DS,DSW,DGM,DGNS,GS}.

Given $\wt\V^\delta_K$ the right hand side $f_K$, restricted as a functional to $\V_K$
may have a part that cannot be seen when restricted to $\wt\V^\delta_K$, so that in general
one misses a {\em data-oscillation error}. We denote by $f_K^\delta$ a suitable projection
to $\wt\V^\delta_K$ which differs from $f_K$ by such a data oscillation. Since under mild assumptions on $f$ these errors are negligible (compared with approximation to $u$) we
ignore such terms in what follows for simplicity of exposition. 
 We can summarize these findings as follows.
\begin{rem}
\label{rem:key}
Validity of \eref{discr-inf-sup} implies that for any $\wg^\delta\in \U^\delta_\cT$
\be
\label{locres}
\|\cF_K(\wg^\delta)\|^2_{\V_K'} \eqsim \|\cF_K(\wg^\delta)\|^2_{(\V^\delta_K)'}
\ee
(up to data oscillation) so that
\be
\label{resPd}
\|\ug- \wg^\delta\|^2_{\U_\cT} \eqsim \|\cF_\cT(\wg^\delta)(v)\|^2_{(\V^\delta_\cT)'}
= \sum_{K\in \cT}\|\mT^\delta_{\cF_K}(\wg^\delta)\|^2_{\V_K} ,\quad \wg^\delta\in \U^\delta_\cT,
\ee
where the quantities $\|\mT^\delta_{\cF_K}(\wg^\delta)\|_{\V_K}$ are now computable.
\end{rem}

\begin{rem}
\label{rem:distinct}
We emphasize that the equivalence \eref{resPd} holds for {\em any} $\wg^\delta\in\U^\delta_\cT$, not just for the current method-projection as
is the case for classical residual error estimator for Galerkin methods, 
see e.g. \cite{Verf}.
This is essential in the present context.
\end{rem}

\begin{rem}
\label{rem:both}
Note that well-posedness of \eref{meshdep} draws on both the bijectivity of $\cB$ and $\cB^*$ underlying the scenarios \eref{eq:sf}, \eref{eq:uwf}.
\end{rem}

Returning to the parameter-dependent case, note first that now, due to the appearance of
the graph norm $\|\cB_\pp\cdot\|_{L_2(\Omega)}$ in the definition of the skeleton, the factor $\U_{\cT,s,\pp}$ may depend now on $\pp$ and so does $\U_{\cT,\pp}:= \U_{\cT,b}\times
\U_{\cT,s,\pp}$. Likewise, the test spaces $\V_{\cT,\pp}=\prod_{K\in\cT}\V_{K,\pp}$
will generally depend on $\pp\in\mD$.
 The roles of $\X, \Y$ from \eref{XY} is then played by
\be
\label{XYT}
\X_\cT:= L_2\bigl(\mD;(\U_{\cT,\pp})_{\pp\in\mD}\bigr),\quad 
\Y_\cT:= L_2\bigl(\mD;(\V_{\cT,\pp})_{\pp\in\mD}\bigr),
\ee
and analogously for their semi-discrete counterparts $\wh\X_\cT, \wh\Y_\cT$.
\begin{rem}
\label{rem:if} 
If the spaces $\X, \Y$ associated with the (parameter-dependent) conforming problem \eref{original} are well-defined as direct integrals, i.e., Property \ref{ass:HilbertIntegral} holds, it is not hard to show that this carries over
to the mesh dependent variants   
mesh-dependent variants $ \X_\cT,  \Y_\cT$ which are also well-defined direct integrals.
In the same manner, Property \ref{ass:bil} carries over to the broken spaces,
so that Theorems \ref{lem:dualHilbertIntegral} and \ref{thm:suff} remain applicable.
\end{rem}

The error-residual relation asserted by Remark \ref{rem:key} holds for any 
piecewise polynomial in the finite-dimensional DPG trial space $\U^\delta_\cT$.
It immediately carries over to the parameter-dependent case as long as
the trial functions are piecewise polynomials in $\U^\delta_\cT$ as functions
on $\Omega$. This motivates the notion of {\em hybrid} hypothesis classes 
$\cH(\Theta)$. By this we mean approximation systems that
  are
comprised of functions $\wg^\delta(x,\pp;\theta)$, $\theta\in\Theta$, that are linear combinations
of piecewise polynomial basis functions $\phi_i(x)$, $i\in \cI_\cT$, with coefficients represented as neural networks with input parameters $\pp\in \mD$ and trainable weight vectors
$\theta$ as outlined in Section \ref{ssec:hybrid}.
Hence, under the given assumptions, Remark \ref{rem:key} implies that 
\be
\label{DPGiloss}
\|\ug - \wg^\delta(\theta)\|^2_{\X_\cT}\eqsim \sum_{K\in \cT}  \int_\mD \|\mT^\delta_{\cF_K}\wg(\cdot,\pp;\theta)\|^2_{\V_{K,\pp}}d\mu(\pp).
\ee
Accordingly,
for any finite training set $\wh\mD\subset\mD$ 
and any $\wg^\delta(\cdot,\pp;\theta) \in \cH(\Theta)\subset \U^\delta_{\cT,\pp}$
one has 
\be
\label{DPGloss}
\frac{1}{\#\wh\mD}\sum_{\pp\in \wh\mD} \|\ug(\cdot,\pp) - \wg^\delta(\cdot,\pp;\theta)
\|^2_{\U_{\cT,\pp}}\eqsim \frac{1}{\#\wh\mD}\sum_{\pp\in \wh\mD}
\sum_{K\in \cT}\|\mT^\delta_{\cF_K}(\wg^\delta)(\cdot,\pp;\theta)\|^2_{\V_{K,\pp}}  ,
\ee
where the right hand side is the desired variationally correct computable loss function.

\section{Network Architectures and Computational Strategies}
\label{sec:arch}

\subsection{Low-rank ResNet hybrid architectures}\label{ssec:DNNhybrids}
We next discuss in further detail the hybrid representation 
format introduced in Section \ref{ssec:hybrid} for approximating solutions $\pp\mapsto u(\pp)$ to \eref{wlB}. Separating spatio-temporal and parametric variables, we consider
  piecewise polynomials as functions of spatio-temporal variables 
with {\em parameter dependent} coefficient functions as in \eqref{eq:hybridformat} that can, for instance, be represented as DNNs.
In doing so we exploit the systematic convergence of finite element methods in small spatial dimensions in combination with reliable concepts for evaluating dual norms.

This format is instrumental for strategy \eref{eq:uwf} to construct computable variationally correct loss functions without the use of  min-max optimization.
It is not strictly necessary for scenario \eref{eq:sf}, but still offers several significant advantages
there as well. First,
rigorously enforcing essential boundary conditions for DNNs is an issue, see \cite{Canuto} for various strategies. 
The hybrid format instead greatly facilitates incorporating essential (spatial) boundary conditions.    Second, since the input parameters of the coefficient neural networks are just the model parameters, training requires only efficient backpropagation, and the neural network activation function need not have higher than first-order weak derivatives. 
When using DNNs with both spatio-temporal and parametric variables 
as input parameters, spatial-parametric mixed derivatives are required. This has to be done with great care to avoid a significant loss of efficiency in the evaluation of gradients with respect to neural network parameters; see, for example, \cite{Pardo}. 

Recall that, as noted in Section \ref{ssec:proj}, we are especially interested in nonlinear approximations that can be more efficient than reduced models based on projection to some fixed linear subspace, such as reduced basis methods. 
This rules out some typical elements of neural network architectures, such as linear output layers, where the output of the previous layers (of length $n$) is multiplied onto some fixed matrix $A^{N\times n}$ to produce the final output, where $N$ is the size of the spatio-temporal discretization. Such a structure is used for parametric problems, for example, in \cite{GPRSK:21}. 
In this case, however, the columns of $A$ play the role of a reduced basis, and the remaining layers of the neural network simply produce coefficients for these basis elements.
Moreover, rather than optimizing heavily overparameterized neural networks as common in machine learning, in our setting we are rather interested in networks that yield data-sparse approximations.

For the spatial or spatio-temporal discretization, we assume a hierarchy of meshes $\cT_1, \ldots, \cT_J$ generated by uniform subdivisions, and we approximate the solution on the finest mesh $\cT_J$. For $j=1,\ldots,J$, we have bases $\{ \phi_i^{(j)} \}_{i \in \mathcal{I}_j}$ of the respective finite element spaces. As outlined in \eqref{eq:hybridformat}, we approximate solutions depending on the parameter $\pp$ in the form
\[
 p\mapsto  \sum_{i \in \mathcal{I}_J} \mathbf{u}_{\mathrm{NN},i}(p; \theta) \phi_i^{(J)},
\]
where the vector $\mathbf{u}_{\mathrm{NN}}(\cdot; \theta)$ is a neural network with trainable parameters $\theta$.

We next disuss the particular network architectures used for the coefficients $\mathbf{u}_{\mathrm{NN}}$, which are based on \emph{residual networks} (ResNets). These are widely used in scientific computing due to their favorable numerical stability properties. 
We use residual networks with layers $\Phi_{(A,W,b)} \colon \R^n\to\R^n$ for $n$ to be determined of the particular form 
\begin{equation}\label{eq:resnetlayer}
	z \,\mapsto\, \Phi_{(A,W,b)}(z) = z+A\sigma( W z + b),
\end{equation}
where $A\in \mathbb R^{n \times r}$, $W\in \mathbb R^{r\times n}$, $b \in\mathbb R^n$, with componentwise application of the activation function $\sigma\colon \R\to\R$. Here, we aim for $r \ll n$ to achieve a data-sparse representation, and by analogy to low-rank matrix representations, we will call $r$ the \emph{rank} of the layer.

With $n_j = \#\mathcal{I}_j$, let the (sparse) prolongation matrices $P_j\in\mathbb R^{n_j,n_{j-1}}$ be given for $j = 1,\ldots,J$ by
\begin{equation}\label{eq:prolongationdef}
    \sum_{i \in \mathcal{I}_{j-1}} \mathbf{v}_i \phi^{(j-1)}_i =
	\sum_{i \in \mathcal{I}_{j}} (P_j \mathbf{v})_i \phi^{(j)}_i
	   , \quad \mathbf{v} \in  \R^{n_{j-1}}\,.
\end{equation}
For any matrix $Q \in R^{n_j,n_{j-1}}$ with the same sparsity pattern as $P_j$, we accordingly introduce the \emph{prolongation layers} $\Pi_{Q} \colon \R^{n_{j-1}} \to \R^{n_j}, \, z\mapsto Q z$.

Starting from a fully connected layer $\Lambda_0$ given by
\[  \R^d  \ni p \mapsto \Lambda_0(p) = \sigma(W_0 p + b_0) \in \R^{n_1}, \]
the eventual neural networks are composed of layers $\Lambda_j \colon \R^{n_j} \to \R^{n_j}$ of the form
\begin{equation}\label{eq:Lambdaj}
   \Lambda_j = \Phi_{(A_{j,L_j}, W_{j,L_j}, b_{j,L_j})} \circ \cdots
    \circ  \Phi_{(A_{j,2}, W_{j,2}, b_{j,2})} 
	\circ \Phi_{(A_{j,1}, W_{j,1}, b_{j,1})} \circ \Pi_{Q_j}
\end{equation}
with some $L_j \in \N$, for $j = 1,\ldots,J$. 
The approximation of $\mathbf{u}_{\mathrm{NN}}$ is then of the form
\[
	\mathbf{u}_{\mathrm{NN}}( \cdot; \theta)
	 = \Lambda_J \circ \cdots \circ \Lambda_1 \circ \Lambda_0,
\]
with trainable parameters
\begin{multline*}
  \theta = (A_{J,L_J}, W_{J,L_J}, b_{J,L_J}, \ldots, A_{J,1}, W_{J,1}, b_{J,1}, Q_J, \ldots, \\ 
   \ldots, A_{1,L_1}, W_{1,L_1}, b_{1,L_1}, \ldots, A_{1,1}, W_{1,1}, b_{1,1}, Q_1, W_0, b_0)	\,.
\end{multline*}
Note that multilevel structure in discretizations is also used, for example, in \cite{CGE:23} with a more specialized architecture for elliptic parametric problems. 

 \subsection{Computational strategies for training}
 
 In the experiments, in each test we use a fixed rank parameter $r$ for all ResNet layers as in \eqref{eq:resnetlayer} and vary the depth of the neural network. We exploit the multilevel structure of the architecture in performing a gradual refinement by subsequently inserting additional ResNet layers into the layer groups $\Lambda_j$ as in \eqref{eq:Lambdaj}, which amounts to increasing $L_j$. These additional ResNet layers are initialized with zero weights. The layer $\Lambda_0$ is initialized with a random initialization with Gaussian entries of vanishing mean and variance $10^{-3}$.
The matrices $Q_1, \ldots, Q_J$ are initialized as $Q_j = P_j$, with $P_j$ as in \eqref{eq:prolongationdef}, for each $j$.
 As common in the context of neural network approximations for PDEs, we use L-BFGS for minimizing the loss functions.

\section{Numerical Experiments}
\subsection{Validation}\label{ssec:validation}
In subsequent experiments we will {\em not}
prescribe the ``exact solution'', typically as an analytically representable smooth
function, but rather discuss scenarios that bring better out the characteristic features 
of the problem under consideration. In the elliptic case this permits low regularity of solutions due to the contrast in the diffusion coefficients. In the case of transport
equations the tests include the appearance of shear layers.
Hence, we do not have the ``ground truth'' at hand by which we could validate numerical results. Instead, we exploit that for each parameter $\pp$ we can apply a finite element
method which we know is based on a stable variational formulation so that corresponding
residuals in the right norm tightly reflect the achieved accuracy with respect 
to the exact solution. 

Just for the sake of illustrating the performance of
the proposed estimation method we compute as ``ground truth''   in addition finite element
solution coefficients $\mathbf{u}_{\rm FE}(\pp_i)$ (with respect to the appropriate problem dependent framework) for 
a set $\wh\mD_{\rm test}\subset \mD$
of $N_{\rm test}$ test samples $\pp_i$. To assess the corresponding  ``generalization error'' of FE-solutions
with respect to the parameters we  evaluate the respective parameter-dependent variationally correct 
loss function  at these test samples. Since we have made sure that the underlying variational formulation is stable we know that these quantities tightly reflect the
error with respect to the respective trial norm $\|\cdot\|_{\U_\pp}$.

As will be seen in more detail later, for the hybrid format
the loss function can, with the notation
\[
  \norm{\mathbf{x}}_{\mathbf G}^2 = \langle \mathbf{x}, \mathbf{G} \mathbf{x} \rangle,
\]
 be represented as $\norm{\mathbf B_{\pp_i} \mathbf u_{\mathrm{FE}}(\pp_i)-\mathbf{f}}_\mathbf{G}^2$, where $\mathbf B_{\pp_i}$ is the representation of the operator $\cB_{\pp_i}$ on the FE space and $\mathbf G$ represents the Riesz lift.
 
The following first validation quantity is then the relative error achieved 
by the FE solution over $\wh\mD_{\rm test}$ in the mean square sense. 
\begin{align}
\label{err1}
\epsref &=\sqrt{ \frac1{N_{\mathrm{test}}} \sum_{i=1}^{N_{\mathrm{test}}} \frac{\norm{\mathbf B_{\pp_i}\mathbf u_{\mathrm{FE}}(\pp_i)-\mathbf{f}}_{\mathbf G}^2}{\norm{\mathbf{f}}_{\mathbf{G}}^2} } ,
\end{align}
Note that the size of this quantity is determined by the spatial resolution of the FE space (provided $N_{\rm test}$ is large enough) and reflects the achievable 
spatial discretization error.

We denote by $u_{\rm net}(\cdot;\pp)$ the result of our training over the given
hybrid hypothesis class which we can now compare with $u_{\rm FE}(\cdot;\pp)$ for 
each test sample $\pp_i\in \wh\mD_{\rm test}$. Rather than computing the norms
of the various solution components in $\X$ (which is a product space) we exploit
that errors in $\X$ are equivalent to residuals in $\Y$, which in turn can be
assessed through the variationally correct loss function. 

Finally, we monitor the generalization error of the prediction $u_{\rm net}$ directly via the
variationally correct loss function, tested over $\wh\mD_{\rm test}$
\be
\label{err3}
\epspred =\sqrt{\frac{1}{N_{\rm test} }\sum_{i=1}^{N_{\rm test}} \frac{\norm{\mathbf B_{\pp_i} \mathbf u_{\mathrm{net}}(\pp_i)-\mathbf{f}}^2_{\mathbf G}}{\norm{\mathbf{f}}^2_{\mathbf G}} }
\ee
This is the quantity that would be evaluated in an application as a certification.
Throughout this section, \emph{level} refers the number $J$ as in Section \ref{ssec:DNNhybrids} of uniform refinements 
used in the spatial discretization starting from two triangles; that is, a triangulation of level $j$ has $2^{j}+1$ grid points in each spatial variable. \emph{Rank} refers to the width of the ResNet layers as in \eqref{eq:resnetlayer}.

With the architectures described in Section \ref{sec:arch}, we consider numerical experiments for two different types of PDEs.
Throughout our tests  for the hybrid models we used a LeakyReLU activation functions $\sigma(x) = \max\{x, 10^{-3} x\}$.
In all experiments (unless indicated otherwise) we used L-BFGS optimization, retaining a maximum of nine previous BFGS updates, with a maximal number of 5000 iteration for the single-parameter case and  20000 iterations for the other experiment. 

\subsection{ An elliptic problem in scenario  (\ref{eq:sf})}
We deliberately choose an example with non-smooth diffusion coefficients that might lead to solutions with low spatial regularity
 generally not belonging to $H^2(\Omega)$. 
Specifically,   we consider \eref{Poisson} with $d_\pp=4$ parameters, where
  $f=1$ and $a(\pp) = \sum_{i=1}^4(\pp_i\alpha_i(x))$, and $\alpha_i$ is the indicator function of one of the four squares of a uniform subdivision. The indicator function are ordered such that $\alpha_1$ is in the top left and $\alpha_4$ is in the bottom right. 

As in \eref{mixed} we consider the corresponding first order operator
$$
\mathcal B_\pp : [\boldsymbol{\sigma}, u]\mapsto \begin{pmatrix}
	\boldsymbol{\sigma} -a(\pp)\nabla u\\ -\operatorname{div}\boldsymbol{\sigma}
\end{pmatrix}
$$ 
which maps $[\boldsymbol{\sigma}, u]\in H(\mathrm{div},\Omega)\times H^1_0(\Omega)=:\U$ to $L_2=
L_2(\Omega;\R^2)\times L_2(\Omega)$. Thus the trial spaces do not depend on the
parameters in an essential way and we can work in scenario \eqref{eq:sf}.
The Riesz lift in this case is the identity, and thus $\mathbf G$ is the scaled identity matrix.
For the spatial discretization, we use piecewise linear elements in $H^1_0(\Omega)$ for $u$ and lowest-order Raviart-Thomas elements in $H(\operatorname{div};\Omega)$ for $\boldsymbol{\sigma}$.

For $N$ training points $\pp_i\in 
\wh\mD\subset \mD$  the loss function reads 
$$
\frac 1 N \sum_{i=1}^N \Bigl( \| \boldsymbol{\sigma}(\pp_i)-a(\pp_i)\nabla u(\pp_i) \|_{(L^2(\Omega))^d}^2+ \| \operatorname{div}\boldsymbol{\sigma}(\pp_i)+ f \|_{L^2(\Omega)}^2 \Bigr).
$$
Note that in \cite{FRO:24,MR:21}, a similar loss function has been used on the spatial domain in a parameter-independent problem, where the flux variable $\boldsymbol{\sigma}$ is interpreted as a `certificate' that leads to a strict upper bound of the error in $u$.

\subsubsection{Numerical results in the hybrid case}\label{ssec:ellnumer}

Table \ref{tab:Elliptic} displays the relative residuals $\epsref$ of the Least-Squares Galerkin (LSG) solution and of $\epspred$ the neural network prediction, where the neural network has 15 layers in total. 
In Figures \ref{fig:snapshotElliptic1} and \ref{fig:snapshotElliptic2}, two snapshots based on the DPG solution and the network prediction are displayed. As one can see, the hybrid method captures the essential features of the FEM solutions and also provides an error reduction under refinement.

\begin{table}[h!]
	\centering
		\begin{tabular}{|c|c|c|}
			\hline
			& Level 3, Rank 40& Level 4, Rank 60  \\
			\hline
			$\epsref$ &0.055&0.028\\
			\hline
			$\epspred$ &0.057 & 0.039 \\
			\hline
		\end{tabular}
	\vspace{3mm}
	\caption{The relative norms of the residuals for two different refinement levels for a network with 15 layers and 1000 training samples}
	\label{tab:Elliptic}
\end{table}

\begin{figure}[h!]
		\centering
	\begin{tabular}{ccc}
		{\footnotesize LSG solution for elliptic case}& {\footnotesize prediction by NN in elliptic case}& {\footnotesize differences of LSG and NN}  \\
\includegraphics[width = 0.3\textwidth]{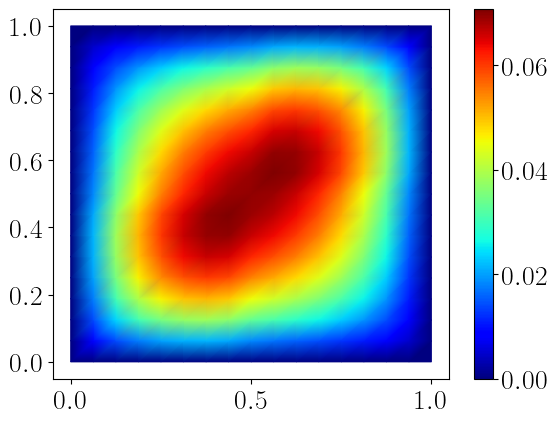}&
\includegraphics[width =0.3\textwidth]{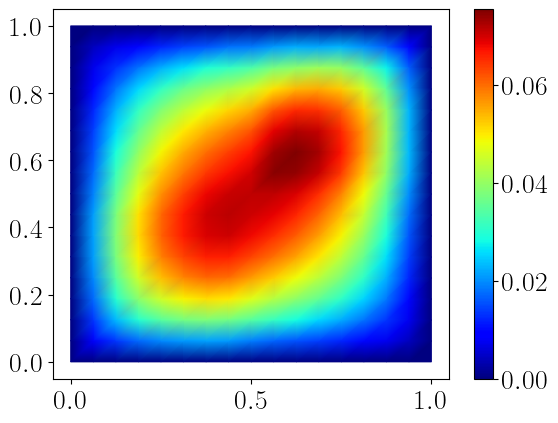}&
		\includegraphics[width =0.305\textwidth]{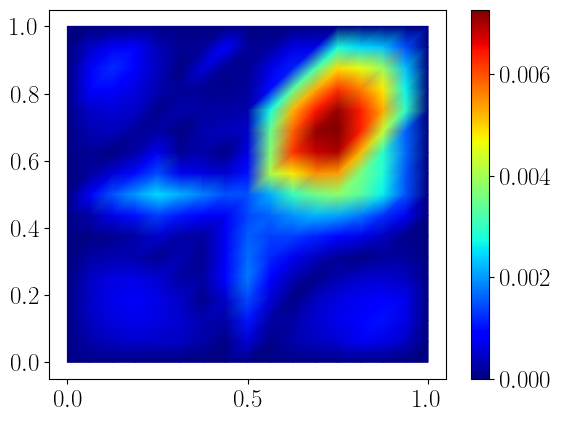}
	\end{tabular}
			
		\caption{Snapshot of the LSG solution, the prediction by the neural network (with rank 60 and 15 layers trained on 1000 samples) and the pointwise absolute value of the difference for a level four refinement evaluated for the parameter $\pp = [0.65,1.45,1.45,0.65]$ corresponding to a checkerboard configuration.}
		\label{fig:snapshotElliptic1}
\end{figure}
\begin{figure}
	\centering
		\begin{tabular}{ccc}
			{\footnotesize LSG solution for elliptic case}& {\footnotesize prediction by NN in elliptic case}& {\footnotesize differences of LSG and NN}  \\
		\includegraphics[width = 0.3\textwidth]{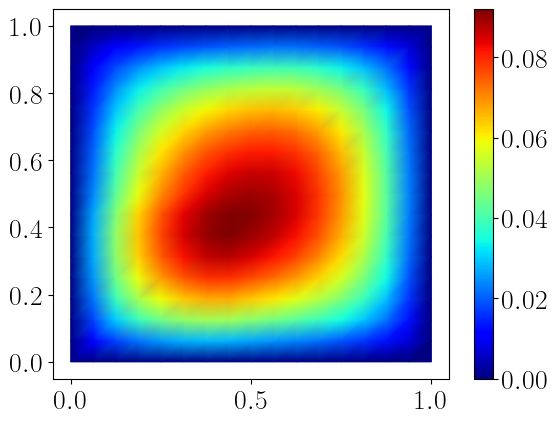}&
		\includegraphics[width =0.3\textwidth]{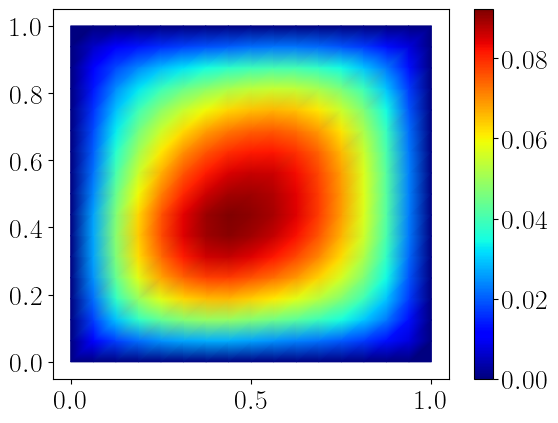}&
		\includegraphics[width =0.305\textwidth]{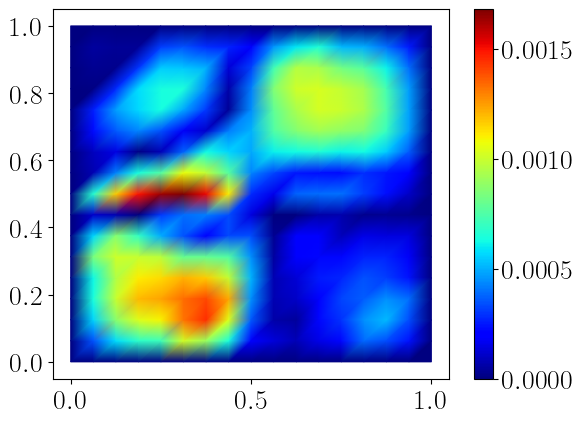}
		\end{tabular}
		\caption{Snapshot of the LSG solution, the prediction by the neural network (with rank 60 and 15 layers trained on 1000 samples) and the pointwise absolute value of the difference for a level four refinement and  given parameter $\pp = [0.53,1.09,0.84,0.82]$ }
		\label{fig:snapshotElliptic2}
\end{figure}

\subsubsection{Comparison to a vanilla PINN approach}
Recall from Remark \ref{rem:PINN} that in the present setting, a standard PINN  approach, applied to the first order system, 
does not incur any variational crime; see, for example, \cite{CaiPINN} for earlier work in this
regard. Of course, one still has to be careful in enforcing
 essential boundary conditions. 
This is in contrast to PINN applied to the original second order formulation.
Regarding boundary conditions, a weighted $\ell_2$ term is
not strictly correct, and we refer to \cite{Bertol,Canuto} for respective
alternative  strategies. We compare below two options.

 To that end, we take a separate neural network for $u$, $\boldsymbol{\sigma}_{1}$ and $\boldsymbol{\sigma}_{2}$ each of them being a fully connected neural network from $\mathbb R^6$ to $\mathbb R$ with a final linear layer, i.e. we have $2+4$ variables in the input and consider the full parametric problem. We used the sigmoid activation function. Furthermore, the networks width maximal with 196 are trained with ADAM with decreasing step sizes from $0.05$ to $0.005$ and a maximum of 2500 iterations. Here, we used the NeuralPDE.jl package \cite{NeuralPDE} in its default setting, that is, we use numerical differentiation for the derivative with respect to the spatial variables. The architectures PINN-B uses an $L_2$-like loss to enforce the boundary conditions, whereas PINN-T uses the trimming function $x(1-x)y(1-y)$ multiplied to the network output to enforce the boundary conditions.

\begin{table}[h!]
	\centering
		\begin{tabular}{|c|c|c|c|c|c|c|}
			\hline
			&H-Net Level 3 & H-Net Level 4& PINN-B-1 & PINN-B-2 & PINN-T-1 & PINN-T-2 \\
			\hline
			  $\epspred$ & 0.057 & 0.039 & 0.05 & 0.046 & 0.036 &0.069\\
			\hline
			dof &26852&  165575& 119364 & 175800 &119364 & 175800\\
			\hline
		\end{tabular}
	\vspace{3mm}
	\caption{The relative dual norm of the residual  in the elliptic problem of the prediction by a hybrid Low-rank ResNet (H-Net) for refinements of level 3 and 4 respectively as well as for a PINN approach with weakly enforced boundaries (PINN-B) and with trimming function (PINN-T) with three and four dense layers, respectively (PINN-T/B-1, PINN-T/B-2)}
	\label{tbl:Pinn}
\end{table}
In Table \ref{tbl:Pinn}, one can see that with a similar number of weights and biases, the hybrid architectures can perform at the same level as the PINNs or even outperform them. It is important to note, however, that the hybrid approximations offer a more systematic way to increase the network size and then actually decrease the loss. In our experiments, the hybrid approach exhibits a significantly enhanced optimization efficacy and predictability.

\subsection{ Linear transport with parameter-dependent convection field in scenario (\ref{eq:uwf})}
\label{sssec:transportnum}
We consider the transport model from Section \ref{sssec:transport}. Note
that now a rigorous interpretation of the continuous reference loss as an expectation
requires   resorting to the concept of direct integrals, see Section \ref{ssec:DPG}. Moreover,
such models exhibit several significant obstructions. First, standard Galerkin or discontinuous Galerkin discretizations are not uniformly inf-sup stable.
Second, one faces a generally non-smooth parameter dependence, as well as an essential dependence of the graph spaces and boundary conditions on the
parameters.
In this case, we find scenario \eqref{eq:uwf} more adequate. In order to highlight related effects,
we consider the simple case of a single parameter in the convection field as well as
the case of a convection field depending on several parameters.

For the spatial discretization, we employ a Discontinuous Petrov Galerkin finite element method with piecewise constant elements on the triangles as well as piecewise linear elements on the skeleton in the trial space and piecewise quadratic elements for the test space as in \cite{BDS}.

Denote by $ \mathbf u_{\rm net} $ the prediction of the neural network and by $ \mathbf u_{\rm DPG}(\pp)$ the DPG solution for given parameter $\pp$. 
Let $\mathbf B_{\pp_i}$ be the representation of the operator $\mathcal B_{\pp_i}=\mathbf b(\cdot,\pp_i)\nabla$ in the finite element spaces, $\mathbf G_{\pp_i}$ be the parameter dependent representation of the Riesz lift and  $\mathbf f$ be the evalutation of $\langle f, \cdot\rangle$ at the basis functions of the discretized test space.  
Then, given training samples $\pp_i$ the loss function reads 
$$\frac 1 N\sum_{i=1}^{N} \norm{\mathbf B_{\pp_i} \mathbf u_{\rm net}(\pp_i)- \mathbf f}^2_{\mathbf G_{\pp_i}}.$$
We analyze the quality of the obtained predictions again 
by monitoring the quantities $\epspred$ and $\epsref$ introduced in Section \ref{ssec:validation} for the above loss function.

 \subsubsection{One Parameter Linear Transport Equation}
First, we   consid  a linear transport equation in two spatial variables parameterized by the angle of the velocity field, i.e. 

$$ [\cos(\pi\pp),\sin(\pi\pp)]\cdot \nabla u + c\ u=f$$
with $$f= \chi_{[0.25,0.5]^2},\quad c=0$$
and  $\pp\in(0,\frac 1 2)$.

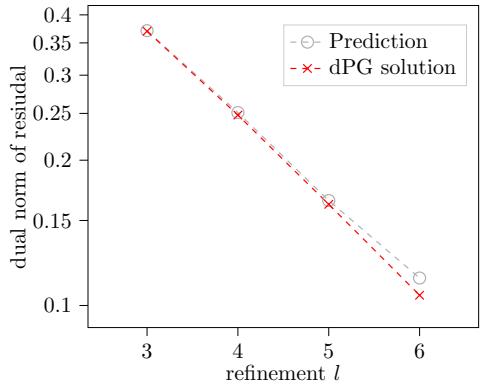
\begin{figure}[h!]
	\centering
\begin{tikzpicture}[scale=0.75]
	
	\definecolor{darkgray176}{RGB}{176,176,176}
	\definecolor{lightgray204}{RGB}{204,204,204}
	\definecolor{steelblue31119180}{RGB}{31,119,180}
	\begin{axis}[
		legend cell align={left},
		legend style={
			fill opacity=0.8,
			draw opacity=1,
			text opacity=1,
			at={(0.97,0.75)},
			anchor=south east,
			draw=lightgray204
		},
		tick align=outside,
		tick pos=left,
		x grid style={darkgray176},
		xlabel={refinement $l$},
		xmin=2.35, xmax=6.65,
		xtick style={color=black},
		y grid style={darkgray176},
		ymin=0.09, ymax=0.418279102007188,
		ylabel={dual norm of resiudal},
		ymode=log,
		ytick style={color=black},
		ytick={ 0.1,0.15,0.2,0.25,0.3,0.35,0.4},
		yticklabels={
			\(\displaystyle {0.1}\),
			\(\displaystyle {0.15}\),
			\(\displaystyle {0.2}\),
			\(\displaystyle {0.25}\),
			\(\displaystyle {0.3}\),
			\(\displaystyle {0.35}\),
			\(\displaystyle {0.4}\)
		}
		]
		\addplot [semithick, darkgray176, dashed, mark=o, mark size=3, mark options={solid}]
		table {%
			3  0.371
			4   0.251
			5   0.165
			6  0.114
		};
		\addlegendentry{Prediction}
		\addplot [semithick, red, dashed, mark=x, mark size=3, mark options={solid}]
		table {%
			3   0.37
			4   0.248
			5   0.162
			6   0.105
		};
		\addlegendentry{dPG solution}
	\end{axis}
	
\end{tikzpicture} 
	\caption{Visualization of the relative dual norm of the residual of the prediction by the neural network versus the DPG solution for a network with rank 20 trained on 1000 samples.}
	\label{fig:academic_example}
\end{figure}

\begin{figure}[h!]
		\centering
			\begin{tabular}{ccc}
				{\footnotesize DPG solution for transport case}& {\footnotesize prediction by NN in transport case}& {\footnotesize differences of DPG and NN}  \\
				\includegraphics[width = 0.3\textwidth]{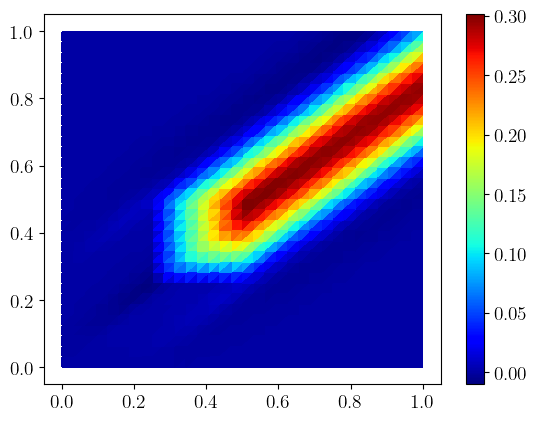}&
			\includegraphics[width =0.3\textwidth]{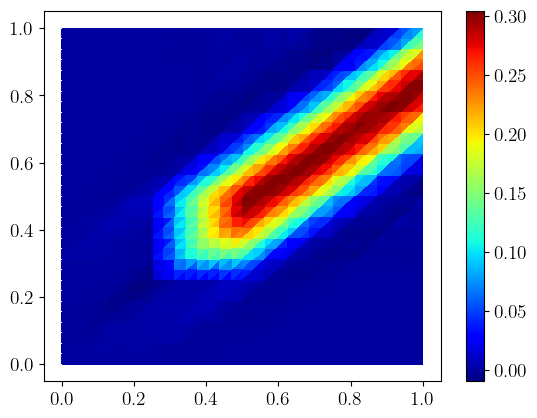}&
			\includegraphics[width =0.31\textwidth]{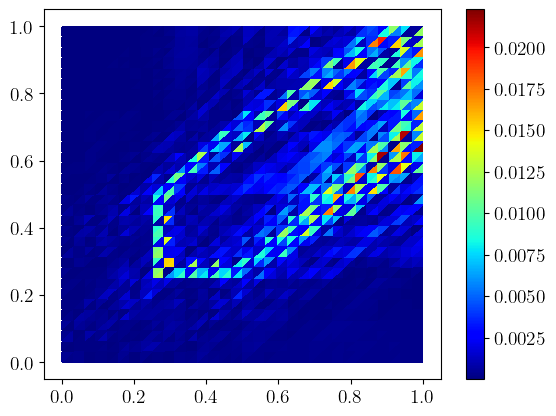}
			\end{tabular}
			\caption{Snapshot for $\pp = 0.2$ of piecewise constant elements for the DPG solution and the prediction by the neural network with rank 20 and 15 layers, trained on 1000 samples for level 6. }\label{fig:snapshotElliptic}
\end{figure}
	\begin{figure}[h!]
		\begin{center}
			\begin{tabular}{ccc}
				{\footnotesize DPG solution for transport case}& {\footnotesize prediction by NN in transport case}& {\footnotesize differences of DPG and NN}  \\
				\includegraphics[width = 0.3\textwidth]{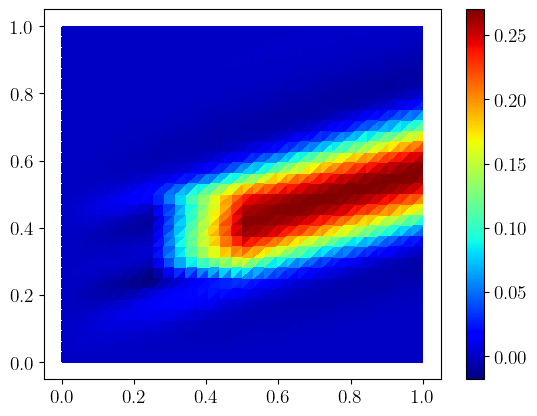}&
			\includegraphics[width = 0.3\textwidth]{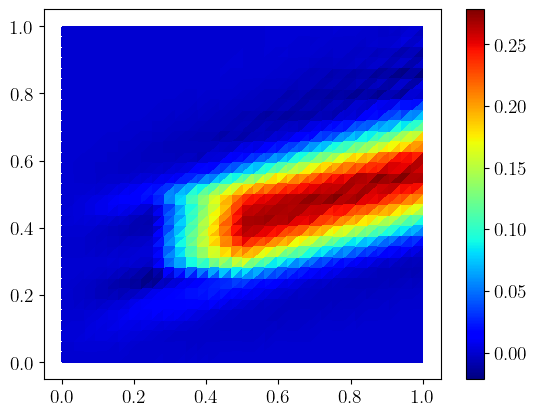}&
			\includegraphics[width =0.31\textwidth]{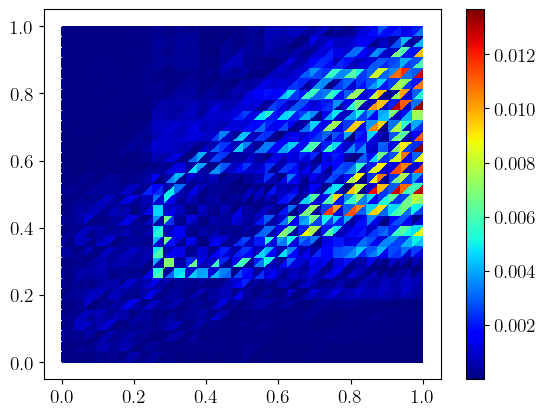}
						\end{tabular}
			
			\caption{Snapshot for $\pp = 0.1$ of piecewise constant elements for the DPG solution and the prediction by the neural network with rank 20 and 15 layers, trained on 1000 samples for level 6. }
		\end{center}
\end{figure}

Table \ref{Tab:err3} presents the accuracy achieved by the predictions. It displays the
values of $\epspred$ for various training settings. The value in the bottom   row
 shows the accuracy $\epsref$ achieved by the individual DPG solutions, see \eref{err1}.
We see that both quantities $\epsref$ and $\epspred$ are very close, reflecting a very good
estimation quality in this case. Of course, the rather coarse spatial discretization,
the very low regularity of the exact solutions and the low order discretization
cause an  overall poor accuracy. 

\begin{table}[h!]
	\centering
		\begin{tabular}{|c|c|c|c|c|}
			\hline
			Rank&Layers& { 1000 samples } &{ 1500 samples }& { 2000 samples }  \\
			\cline{1-5}
			\multirow{2}{*}{ 10}&11   &0.263 & 0.2618 & 0.262  \\
			&15  & 0.252  & 0.255 &  0.258\\
			\cline{1-5}
			\multirow{2}{*}{15}&11   & 0.251  & 0.252 & 0.255\\
			&15&   0.25   & 0.25  & 0.25\\
			\cline{1-5}
			\multirow{2}{*}{20}&11  & 0.253  & 0.253  & 0.253\\ 
			&15  & 0.25  & 0.25  & 0.25\\
			\cline{1-5}
			\multirow{2}{*}{30}&11 & 0.251  & 0.249 & 0.252\\ 
			&15  & 0.249  & 0.249  & 0.249\\
			\hhline{|=|=|=|=|=|}
			\multicolumn{4}{|c|}{relative DPG residual $\epsref$}&0.248\\
			\hline
			
		\end{tabular}
		\vspace{3mm}
		\caption{$\epspred$: The relative dual norm of the residual of the prediction by a hybrid low-rank ResNet for the transport problem with one parameter and a  refinement of level 4}\label{Tab:err3}
\end{table}

 \subsection{Multi-parametric linear transport with piecewise linear vector field}

Second, we  consider  a linear transport equation in two physical dimensions, that is, one temporal and one spatial dimension, where the convection field is parameterized by hat functions in space with coefficients constant in time. The problem thus takes the form
\[  \biggl[1,\,\sum_{i=1}^7 a_i(x) \pp_i \biggr] \cdot \nabla_{t,x} u + c\,u=f \]
with  
\[   f= \chi_{[0.25,0.5]^2},\quad c=0, \]
where $a_i$ are hat functions in space constant in time on the interior nodes of a grid with 9 nodes and $\pp_i\in [0.3,1.3]$. We use the same error measures as above in \eref{err1} and \eref{err3}. 

Table \ref{Tab7-4:err3} shows the results for $\epsref$ and $\epspred$ of analogous 
tests for the case of $7$ parameters, in addition to the two spatio-temporal dimensions. In principle, the results are very similar, except for 
a slightly lower agreement between the trained prediction and the DPG solutions evaluated at the test samples. In view of the high dimensionality, this is perhaps not surprising. We observe that increasing the rank does not necessarily improve the achieved accuracy. However, this is likely due to the increased complexity of the optimization task, where it is generally difficult to assess whether a given optimization schedule stalls on a mediocre plateau or has converged.   
\begin{table}[h!]
		\begin{tabular}{|c|c|c|c|}
			\hline
			Rank &11 Layer &13 Layer &15 Layer  \\
			\hline
			10 & 0.454 &0.447 &0.44  \\
			\cline{1-4}
			15  & 0.442& 0.431 & 0.428  \\
			\cline{1-4}
			{20}&0.481 & 0.464& 0.455   \\ 
			\hhline{|=|=|=|=|}
			\multicolumn{3}{|c|}{relative DPG residual $\epsref$} & 0.337\\\hline
		\end{tabular}
\vspace{3mm}
	\caption{$\epspred$: The relative dual norm of the residual of the prediction by a hybrid low-rank ResNet for the transport equation with seven parameters and a refinement of level four.}
	\label{Tab7-4:err3}
\end{table}

 \section{Summary and Concluding Remarks}\label{sec:concl}

Our focus in this work has been the approximation of solution manifolds of different classes of parameter-dependent PDEs, with dissipative elliptic or parabolic problems as well as linear transport problems as particular examples. 
A central issue when using nonlinear model classes, such as neural networks, for such approximations is how one can assess the quality of approximate solutions using only computable information. Such control of errors is of crucial importance for surrogate models produced by operator learning, where known training methods lack reliability.

Our approach to controlling errors centers on what we call \emph{variationally correct loss functions}. That is, at any stage of an optimization process,
the current magnitude of the loss is a tight bound for the error in a model-compliant norm, without imposing additional regularity requirements, while at the same time the loss remains computationally realizable. Based on reformulations as first order systems of PDEs, we have considered formulations where the test space is an $L_2$-space, as well as discontinuous Petrov-Galerkin (DPG) methods based on ultra-weak formulations with $L_2$-spaces as trial spaces. Such formulations are available for wide classes of problems beyond the examples considered here.

The common conceptual starting point is an ideal residual loss resulting from the reformulation of the {\em family} of parameter dependent problems as a {\em single} well-posed variational problem. Analyzing such losses, as limits of empirical versions based on parameter samples, involves integrals over the parameter domain, also in cases where
trial and test spaces differ essentially under parameter variation. In order
to still obtain  a generally applicable notion of measurable parameter-dependent functions, leading to expectations as limit loss, we resort to the concept of {\em direct integrals} of Hilbert spaces extending the notion of Bochner spaces.
The concepts put forward here also provide a basis for treating parametric nonlinear PDEs, which will be a subject of further work.

To make the computation of these residual loss functions possible in particular for the DPG-type methods, we have considered a hybrid approximation format combining finite element discretizations on the spatio-temporal domain with neural network approximation of the parameter dependence. Compared to a neural network approximation in all variables (as in standard PINN or operator learning approaches), such a neural network approximation of finite element basis coefficients as functions on the parameter domain has several practical advantages. However, it requires neural networks with large output dimension -- as we have noted, this imposes some strong requirements on possible sparse network architectures. With the particular architectures that we use, we may avoid a curse of dimensionality for large parametric dimensions, a point that requires further investigation.

Aspects that we have not considered here are the approximation of the ideal loss by the corresponding empirical risk as in \eqref{deviation}, which depends also on the process of choosing the samples, and the related problem of estimating generalization errors. These question will be addressed in more detail in future work. 

Our numerical experiments are only intended to provide a proof of concept on comparably small-scale examples for the combination of variationally correct residual losses with hybrid approximations.   
In the test cases, due to the proportionality between residuals and error norms, we can assess the total solution error at each sampling point also in the absence of a reference solution, and thus, rather than resorting to manufactured solutions, we can choose examples reflecting the typical difficulties arising in each problem.
As we have noted, optimization strategies commonly used in machine learning are not necessarily the most effective for the regression problems that need to be solved here, and we thus use quasi-Newton methods in our numerical tests. The further study of solvers is another direction of future work.

In summary, we can conclude that the proposed mechanisms can enable accuracy control in approximating solution manifolds of parameter-dependent problems using nonlinear model classes such as neural networks. 
This is of particular importance in the application of such surrogate models
for the underlying solution operators in regression problems in state estimation and parameter identification problems.

\bibliographystyle{plain}
\bibliography{VC-reg}

\appendix
\section{Proof of Lemma \ref{lem:V_pDirectIntegral}}
\label{appendix:A}
The proof of Lemma \ref{lem:V_pDirectIntegral} will be done in several steps whose
proofs, in turn, use standard techniques concerning denseness results. 
For the convenience of the reader we present here a coherent version.
Throughout this section we assume that Assumption \ref{ass:vectorfield} is valid.

The first step ensures denseness of smooth functions in $\G^*_\pp$ for elements that have support compactly contained in $\Omega$. From \cite[Lemma II.1(i)]{DiPerna}, we obtain the following result.
\begin{lemma} 
\label{lem:MollifieronSubdomain} 
	Let $\psi_\e(x)=\e^{-d}\psi\big(\e^{-1} {x}\big)$ with 
$\psi\in C^\infty(\R^d)$, ${\rm supp}\,\psi\subseteq B_1$ (where in general
$B_\tau$ denotes the ball of radius $\tau$ in $\R^d$) be a mollifier,    and let $v\in\G^*_\pp$.  Let $\Omega' $ be any subdomain with compact closure in $\Omega$. 
Then, for $\e \le \e(\Omega'):= {\rm dist}(\Omega',\partial\Omega)$ the convolution 
 $ \psi_\e * v$ is well defined on $\Omega'$ and belongs to $\G^*_\pp(\Omega'):= \G^*_\pp|_{\Omega'}$, with the canonical understanding of $\|w\|^2_{\G^*_\pp(\Omega')}:= \|w\|^2_{L_2(\Omega')} + \|\cB^*_\pp w\|^2_{L_2(\Omega')}$. Moreover, 
\be
\label{conv}
\|\cB^*_\pp(\psi_\e * v - v)\|_{L_2(\Omega')}\to 0 \quad\text{as}\quad \e \to 0.
\ee
\end{lemma}

As a second step we we use that $\V_\pp\subset \G^*_\pp$, $\pp\in\mD$, and that
we can construct a $\pp$-{\em independent} fundamental system of sections for 
$\prod_{\pp\in\mD}\G^*_\pp$ due to the fact that $H^1(\Omega)$ is dense in $\G^*_\pp$
for each $\pp$, as shown next.

\begin{lemma}
\label{lem:H1}
Assume that $\Omega\subset \R^d$ is a simply connected Lipschitz domain.\\[1.5mm]
(a) The embedding $H^1(\Omega)\subset \G^*_\pp$ is continuous, that is, there exists a $C<\infty$
		such that for all $v\in H^1(\Omega)$ and every $\pp\in\mD$
		\be
		\label{cont}
		\|v\|_{\G^*_\pp}\le C\|v\|_{H^1(\Omega)}. %
		\ee
\smallskip
(b)
   $H^1(\Omega)$ is dense in $\G^*_\pp $ for each $\pp\in\mD$.
\end{lemma}
\begin{proof}
When reference to $\Omega$ is clear we sometimes write briefly $\|\cdot\|:= \|\cdot\|_{L_2(\Omega)}$.
Regarding (a),   observe that $\|\mathcal B^*_\pp v\|^2 \leq (\|c\|^2_{\infty}+\|\mathrm {div}(\mathbf b)\|^2_\infty)\|v\|^2+\|\mathbf b\|^2_{\infty}\|\nabla v\|^2$. Choose $C = \max\{1,\|c\|^2_{\infty}+\|\mathrm {div}(\mathbf b)\|^2_\infty,\|\mathbf b\|^2_{\infty}\}$ (note that since $\mathbf{b} \in W^{1}_\infty$, we have ${\rm div}\,\mathbf b(\cdot;\pp)\in
	L_\infty(\Omega)$) and the claim follows.
For (b), we use an immediate adaptation of the standard argument given in \cite[Thm.~3.17]{AdamsFournier}, based on the partition of unity from \cite[Thm.~3.15]{AdamsFournier}.

\end{proof}

We are now prepared to show Lemma \ref{lem:V_pDirectIntegral} by constructing  a fundamental system of sections for $\prod_{\pp\in\mD}\V_\pp$.

\begin{proof}[Proof of Lemma \ref{lem:V_pDirectIntegral}]
The first ingredient is a $\pp$-dependent cutoff function that reduces smooth functions to elements
in $\V_p$. To that end, it is convenient to abbreviate in what follows the directional derivatives
$D_\pp v:= \bb(\cdot;\pp)\cdot \nabla v$.
Consider the boundary value problem
\be
\label{clip1}
-D_\pp \phi(\cdot;\pp)=1,\quad \mbox{in $\Omega$,\quad $\phi|_{\Gamma_{\rm out}(\pp)}=0$}.
\ee
The solution has an explicit representation that we describe next.
To this end, consider the $\pp$-dependent family of characteristics $z(t)=z(t;x_0;\pp), t\in\R$,
defined by the initial value problems
\be
\label{char}
\dot{z}(t)=  - \mathbf b(z(t;x_0;\pp);\pp),\quad z(0,x_0;\pp) = x_0.
\ee
The Lipschitz properties on the convection field in Assumption \ref{ass:vectorfield}  ensure the unique existence of trajectories for all times. 

 As a next step we   infer then
measurability of the solution $\phi(\cdot;\pp)$ to \eref{clip1} from these properties of the characteristics. These will be based on an explicit representation of $\phi(\cdot;\pp)$.

To describe the solution $\phi(\cdot;\pp)$ to \ref{clip1} in terms of the characteristics,  let the time of escape from a point $x\in \Omega$   along a characteristic path
be
\begin{equation}
	\label{eq::ellDefinition}
	t (x;\pp) := \inf\{ s>0 :   z(s;x;\pp) \notin \Omega\}, \quad x_+(x;\pp)= z(t(x;\pp);x;\pp)\in \Gamma_{\rm out}(\pp).
\end{equation}
For a given $\pp\in\mD$, we can express any $x\in \Omega$ in terms of characteristic coordinates as 
\begin{align}
	\label{charcord}
	x &= z(-t (x;\pp);x_+(x,\pp);\pp),
\end{align}
i.e., 
$x_+(x,\pp)\in \Gamma_{\rm out}(\pp)$ is the intersection of the curve $\{z(t;x;\pp): t\ge 0\}$, with $\Gamma_{\rm out}(\pp)$. %
In these terms,  
\begin{align}
	\phi (x,\pp):= t(x;\pp)
\end{align}
is the time spent on a characteristic path from $x$ in $\Omega$ when traveling with
constant speed one, i.e., the length of the corresponding characteristic segment.
Hence,
\[
D_\pp \phi(x;\pp)= \lim_{h\to 0}\frac  1h\Big(\phi(z(h;x;\pp) ;\pp)-\phi(x;\pp)\Big)= \lim_{h\to 0}\frac 1h\Big(t(z(h;x;\pp) ;\pp)-t(x;\pp)\Big)= 1 
\]
which means that $\phi(x;\pp)$ solves \eref{clip1} $\mu$-a.e.\ in $\mD$. Thus measurability of $\mathbf b(\cdot;\pp)$ implies 
measurability of $z(\cdot;\cdot;\pp)$, which in turn implies measurability of
$t(\cdot;\pp)$ and hence that of $\phi(\cdot;\pp)$. Moreover, by the non-degeneracy assumption on the convection field, for any $x_0\in \Gamma_{\rm out}(\pp)$ we have 
\be\label{P1}
\tag{P1}
|\phi(z(-s;x_0;\pp);\pp)|\eqsim s,\quad s\le t_{x_0},
\ee
where $t_{x_0}$ is the maximal travel time when starting at $x_0\in\Gamma_{\rm out}(\pp)$. Hence, we get 
\[ \phi(\cdot,\pp)|_{\Gamma_{\rm out}(\pp)}=0. \]
Moreover, there is a constant $M$ (depending on the convection field and $\Omega$) such that
\be\label{P2}
\tag{P2} \|\phi(\cdot;\pp)\|_{L_\infty(\Omega)}\le M, \quad  \|D_\pp\phi(\cdot;\pp)\|_{L_\infty(\Omega)}\le M.
\ee

As a second step,  assume that $\{\zeta_n\}_{n\in\N}$ is total in $H^1(\Omega)$ and hence in $\G(\cB_\pp^*)$
for every $\pp\in\mD$. 
To proceed let $C^\infty_{0,\pp}(\Omega) = \{v\in C^\infty(\Omega): {\rm supp}\,v\cap \Gamma_{\rm out}(\pp)
=\emptyset\}$. By the same reasoning as in the proof of Lemma \ref{lem:H1} (b), one sees
that
$C^\infty_{0,\pp}(\Omega)$ is dense in $\V_\pp$. Thus,  it suffices 
to show that the system $\phi(\cdot;\pp)\zeta_n$, $n\in \N$, is dense in $C^\infty_{0,\pp}(\Omega)$. For any given $v\in C^\infty_{0,\pp}(\Omega)$ we have
\begin{align*}
	v- \sum_{n\in\N} \alpha_n \phi(\pp)\zeta_n &= \phi(\cdot;\pp)\Big(\phi(\cdot;\pp)^{-1}v - \sum_{n\in\N}\alpha_n\zeta_n\Big).
\end{align*}
Thus, by \eref{P2},
\begin{align*}
	\Big\|D_\pp\Big( v- \sum_{n\in\N} \alpha_n \phi(\cdot;\pp)\zeta_n\Big)\Big\|_{L_2(\Omega)}
	&\le \Big\|(D_\pp\phi(\cdot;\pp))\Big(\phi(\cdot;\pp)^{-1}v - \sum_{n\in\N}\alpha_n\zeta_n\Big)
	\Big\|_{L_2(\Omega)}\\
	&\qquad + \Big\|\phi(\cdot;\pp) D_\pp \Big(\phi(\cdot;\pp)^{-1} v- \sum_{n\in\N} \alpha_n  \zeta_n\Big)\Big\|_{L_2(\Omega)}\\
	&\le %
	M\big(\Big\|\Big(   \phi(\cdot;\pp)^{-1} v- \sum_{n\in\N} \alpha_n  \zeta_n \Big)\Big\|_{L_2(\Omega)}\\
	&\qquad +\Big\|D_\pp\Big(   \phi(\cdot;\pp)^{-1} v- \sum_{n\in\N} \alpha_n  \zeta_n \Big)\Big\|_{L_2(\Omega)}\big).
\end{align*}
Now we will show that $\phi(\cdot,\pp)^{-1}v\in\G^*_\pp$. 
For $v\in C^\infty_{0,\pp}(\Omega)$ there is a $\delta>0$ and neighborhood $\cN_\delta$ of width $\delta $ such that
$\phi(\cdot;\pp)^{-1}v$ remains bounded in $\Omega\setminus \cN_\delta$ because by \eref{P1}
$\|\phi(\cdot;\pp)^{-1}\|_{L_\infty(\Omega\setminus \cN_\delta)}\lesssim \delta^{-1}$. Since $\phi(\cdot;\pp)$ solves \eref{clip1} and $|\phi(\cdot;\pp)|_{\Omega\setminus\cN_\delta}$ remains
lower bounded by a constant times $\delta^{-1}$ we get by product rule that $\|D_\pp \phi(\cdot;\pp)^{-1}\|_{L_\infty(\Omega\setminus \cN_\delta)}<\infty$. 
 Therefore, $\phi^{-1}(\cdot;\pp)v\in \G^*_\pp$ for each $\pp\in\mD$. As $\xi_n$ are chosen to be total in $\G^*_\pp$ for all $\pp\in\G^*_\pp$  we can choose the $\alpha_n$ to make the term $ \Big\| D_\pp\Big(  \phi(\cdot;\pp)^{-1} v- \sum_{n\in\N} \alpha_n  \zeta_n \Big)\Big\|_{L_2(\Omega)}$ as small as we wish and thus $\phi(\cdot,\pp)\zeta_n$ is total in $\V_\pp$.

Since the $\zeta_n$ are independent of $\pp$ and in $H^1(\Omega)$, and $\phi$ is jointly measurable in $x$ and $\pp$, the
functions $\xi_n = \phi\zeta_n$ are measurable on $\Omega\times\mD$.
 Now consider the mapping
	$$
	\pp\mapsto \langle \xi_n,\xi_m\rangle_{\V_\pp}= \langle \cB^*_\pp\xi_n,\cB^*_\pp \xi_m\rangle_{L_2(\Omega)},
	$$
	which is measurable since by product rule the map $$(x,\pp)\mapsto \mathcal B^*_\pp \xi_n(p,x) = \divv (\mathbf b(x;\pp))\phi(x,\pp)\zeta_n(x)+\mathbf b(x,\pp)\cdot(\zeta_n(x)\nabla \phi(x,\pp)+\phi(x,\pp)\nabla \zeta_n(x))$$ is measurable by assumption on $\mathbf b$ and the measurability of $\phi,\zeta_n,D_\pp \phi$ and $ D_\pp \zeta_n$.  
	Hence, we have found a fundamental sequence of $\mu$-measurable sections in 
	$\prod_{\pp\in\mD}\V_\pp$ so that Theorem 2.2 applies and confirms 
	that $\Y=L_2\big(\mD;(\V_\PP)_{\pp\in\mD}\big)$ is a well-defined direct integral
	with norm $\|v\|_\Y^2= \int_\mD \|v\|_{\V_\pp}^2d\mu(\pp)$.
	
Moreover, $\|\cdot\|_{L_2(\Omega)}\le C\|\cdot\|_{\V_\pp}$ (see the comment following
Assumption \ref{ass:vectorfield}) implies that $\Y$ is continuously embedded in 
$L_2(\mD\times\Omega)=L_2(\mD;L_2(\Omega))$, and, as a Hilbert space, is closed.
Here it is understood that
	$\mD\times \Omega$ is viewed as measure space with the product measure whose 
	marginals are $\mu$ and the Lebesgue measure, respectively.
This completes the proof of
	Lemma \ref{lem:V_pDirectIntegral}. \end{proof}

\section{Additional Remarks on Direct Integrals}
 \begin{example}
 \label{ex:B}
	Let $\Gamma_1$, $\Gamma_2$ be two distinct parts of $\Gamma = \partial \Omega$ for some bounded, simply connected domain $\Omega\subset\R^d$. Define the space $$H_{\Gamma_i} = \{v\in H^1(\Omega): v|_{\Gamma_i}=0\}.$$
	Let $\mD\subset\mathbb R^{d_\pp}$ be measurable and consider $\mD^1\subset \mD$ non-measurable such that all measurable subsets have vanishing measure, such as a Vitali set. Note that $\mD^2 = \mD\setminus\mD^1$ is then also non-measurable. Define $$\W_\pp = \begin{cases}
		H_{\Gamma_1} \quad \text{if }\pp\in\mD^1,\\
		H_{\Gamma_2}\quad\text{else.}
	\end{cases}$$
	Let $(\xi_n(\pp))_{n\in\N}$ be an orthonormal basis in $\W_\pp$. Observe that we can choose $ \xi_n(\pp) =\xi_n^{1,2}$ to be constant when restricted to $\mD^1$ and $\mD^2$, respectively.
	
	 Since by construction
	 $$
	 \pp\mapsto \langle\xi_n(\pp),\xi_m(\pp)\rangle_{H^1(\Omega)}=\delta_{m,n},
	 $$ 
	  $(\xi_n)_{n\in\N}$ is a $\mu$-measurable sequence so that the direct integral is well-defined. Observe, $v\in\prod_{\pp \in \mD}\W_\pp$, with  $v(\pp)=\sum_{n\in\N}v_n(\pp)\xi_n(\pp)$ is $\mu$-measurable if and only if  for all $n\in\N$ the coefficient map 
	  $$
	  \pp\mapsto v_n(\pp)
	  $$ 
is measurable.
	
	Choose $v\in L_2(\mD,(\W_\pp)_{\pp\in\mD})$  as $v(\pp)=\xi_1(\pp)$, i.e. $v_1(\pp)=1$ and $v_n(\pp)=0$ for all $n\geq2$ and $\pp\in \mD$. Then for all $n\in\N$ the coefficient map is constant and hence measurable and square integrable. Since $\bigcap_{\pp\in\mD}\W_\pp$ is not dense in any $\W_\pp$, we can assume (after reordering of the basis $\xi_n^{1,2}$) that $\xi_1^{1,2}\notin  \bigcap_{\pp\in\mD}\W_\pp$ and thus for all $\pp\in \mD$  we have $v(\pp)\notin \bigcap_{\pp\in\mD}\W_\pp$.  Furthermore, let $\tilde \xi_n$ be an orthonormal  basis of $H^1(\Omega)$  with $\tilde\xi_1 = \xi_1(\pp)$ for some $\pp\in\mD^1$, and thus a fundamental system of $\prod_{\pp\in\mD}H^1(\Omega)$.

	Then we have
	for all $\pp\in\mD^1$ that $$\langle v(\pp),\tilde\xi_1\rangle_{H^1(\Omega)} = \langle \xi_1(\pp),\tilde\xi_1\rangle_{H^1(\Omega)}=\langle \tilde\xi_1,\tilde\xi_1\rangle_{H^1(\Omega)}=1$$ as well as 
	for all $\tilde\pp\in\mD^2$ we get
	$$ \langle v(\tilde\pp),\tilde \xi_1\rangle  = \langle \xi_1(\tilde\pp),\tilde\xi_1\rangle_{H^1(\Omega)}<1 $$
	as $\xi_1(\pp)\notin H_{\Gamma_2} $ for $\pp\in \mD^1$ and thus $\langle \xi_1(\pp_1),\xi_1(\pp_2)\rangle<1$ for $\pp_1\in\mD^1$ and $\pp_2\in\mD^2$ (recall that $\xi(\pp)$ is an orthonormal basis for all $\pp\in\mD$).
	Then $$\Psi:\pp\mapsto \langle v(\pp),\tilde\xi_1\rangle$$
	is not measurable since $\Psi^{-1}(\{1\})=\mD^1$ and thus $v\notin L_2(\mD,(H^1(\Omega))_{\pp\in\mD})$.
	
\end{example}
However, if there is a fundamental sequence in $\prod_{\pp \in \mD}\V_\pp$ that is measurable in $\prod_{\pp\in\mD}\W_\pp$, we have the subspace relations we expect.
\begin{prop}
\label{prop:subspace}
	Let $\W_\pp$ be a separable Hilbert space such that $\bigcap_{\pp\in\mD}\W_\pp$ is dense in each $\W_\pp$ and equip it for each $p$ with the same fundamental sequence $(\tilde \xi_n)_{n\in\N}$ with $\tilde\xi_n\in\bigcap_{\pp\in\mD}\W_\pp$ for all $n\in\N$. Consider closed subspaces $\V_\pp$ endowed with the scalar product of $\W_\pp$. Assume that there is a fundamental sequence $(\xi_n)_{n\in\N}$ of $\prod_{\pp\in\mD}\V_\pp$ that is $\mu$-measurable with respect to $(\tilde \xi_n)_{n\in\N}$. Then we have
	\begin{equation}\label{eq:dssubspace}
	 L_2(\mD,(\V_\pp)_{\pp\in\mD})\subset L_2(\mD,(\W_\pp)_{\pp\in\mD})\quad \text{as a closed subspace.}
	 \end{equation}
	Conversely, if \eqref{eq:dssubspace} holds, then the fundamental sequence of $\prod_{\pp\in\mD}\V_\pp$ is $\mu$-measurable in $\prod_{\pp\in\mD}\W_\pp$.
\end{prop}

\begin{proof}
	By definition we have $\prod_{\pp \in \mD}\V_\pp\subset \prod_{\pp \in \mD}\W_\pp$ as sets. By \cite[Lemma~8.12]{Takesaki} we can assume that $(\xi_n(\pp))_{n\in\N}$ forms an orthonormal basis in $\V_\pp$ for all $\pp\in\mD$ and thus $\mu$-measurability in $\prod_{\pp \in \mD}\V_\pp$ reduces to measurability of the coefficient map.  As the map $$\pp\mapsto \langle \xi_n(\pp),\tilde\xi_m\rangle_{\W_\pp}$$ is measurable for all $n,m\in\N$ we have that any $v\in\prod_{\pp\in\mD}$ that is $\mu$-measurable with respect to $\xi_n$ is also $\mu$-measurable with respect to $\tilde\xi_n$ as the coefficient maps are measurable. As the fiber spaces share the same bilinear form and as both Hilbert integrals are complete, we get by definition $$L_2(\mD,(\V_\pp)_{\pp\in\mD})\subset L_2(\mD,(\W_\pp)_{\pp\in\mD})$$ as a closed subspace.
	
	On the other hand, if $L_2(\mD,(\V_\pp)_{\pp\in\mD})\subset L_2(\mD,(\W_\pp)_{\pp\in\mD})$ we get by definition of the Hilbert integral that $$\pp\mapsto \langle \xi_n(\pp),\tilde\xi_m\rangle_{\W_\pp}$$ has to be measurable.
\end{proof}

\end{document}